\documentclass[11pt]{article}
\usepackage{amsmath,amsthm,amssymb,amscd}

\font\Bbb=msbm10
\def\R{\hbox{\Bbb R}}

\newtheorem{theorem}{Theorem}[section]
\newtheorem{lemma}[theorem]{Lemma}
\newtheorem{proposition}[theorem]{Proposition}
\newtheorem{corollary}[theorem]{Corollary}
\newtheorem{definition}[theorem]{Definition}
\newtheorem{remark}[theorem]{Remark}
\newtheorem{example}[theorem]{Example}

\usepackage{epsfig}

\begin{document}

\title{On Gauss codes of virtual doodles}

\author{Andrew Bartholomew \\ 
School of Mathematical Sciences, University of Sussex\\
Falmer, Brighton, BN1 9RH, England\\
e-mail address: andrewb@layer8.co.uk \\ 
\\
Roger Fenn \\ 
School of Mathematical Sciences, University of Sussex\\
Falmer, Brighton, BN1 9RH, England\\
e-mail address: rogerf@sussex.ac.uk \\ 
\\ 
Naoko Kamada \\ 
Graduate School of Natural Sciences, Nagoya City University\\
Nagoya, Aichi 467-8501, Japan\\
e-mail address: kamada@nsc.nagoya-cu.ac.jp \\ 
\\ 
Seiichi Kamada \\ 
Department of Mathematics, Osaka City University\\
Osaka, Osaka 558-8585, Japan\\
e-mail address: skamada@sci.osaka-cu.ac.jp
} 

\maketitle

\begin{abstract} 
We discuss Gauss codes of virtual diagrams and virtual doodles. The notion of a left canonical Gauss code is introduced and it is shown that 
oriented virtual doodles are uniquely presented by left canonical Gauss codes.

\end{abstract}



\section{Introduction} 

A {\it virtual diagram} is a generically immersed $1$-manifold in $\R^2$ such that 
the crossings are classified into two families, real crossings and virtual crossings. 
A virtual crossing is depicted by a crossing encircled with a small circle. 
Throughout this paper, we only consider a case that the $1$-manifold is a circle and we assume that 
a virtual diagram is oriented, as a $1$-manifold. 

For a virtual diagram $D$ with $n$ $(>0)$  real crossings, removing an open regular neibourhood of the real crossings, we have a collection of $2n$ immersed arcs. We call them {\it semiarcs} of $D$.  If two virtual diagrams $D$ and $D'$ are identical except for a semiarc of $D$ and a semiarc of $D'$, then we say that $D'$ is obtained from $D$ by a {\it detour move}.  

Two virtual diagrams are called {\it strictly equivalent} if they are related by a finite sequence of detour moves and isotopies of $\R^2$.  Two virtual diagrams are strictly equivalent if and only if they are related by a finite sequence of moves depicted in Figures~\ref{figVR_1figVR_2} and  \ref{figVR_3figVR_4} and isotopies of $\R^2$ (cf. \cite{BFKK, K}).  The moves are referred to as $VR_1$, $VR_2$, $VR_3$ and $VR_4$, which are flat versions of virtual Reidemeister moves in virtual knot theory \cite{K}.

    \begin{figure}[h]
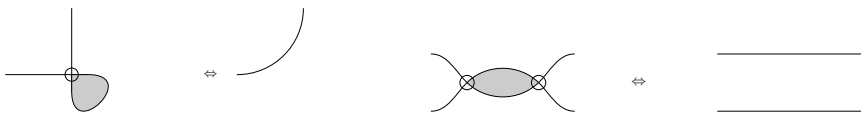

    \centerline{\epsfig{file=doodlefigVR_1.eps, height=1.4cm}
    \qquad \qquad 
    \epsfig{file=doodlefigVR_2.eps, height=0.8cm}
    }
    \vspace*{8pt}
    \caption{Moves $VR_1$ (Left) and  $VR_2$ (Right)}\label{figVR_1figVR_2}
    \end{figure}

    \begin{figure}[h]
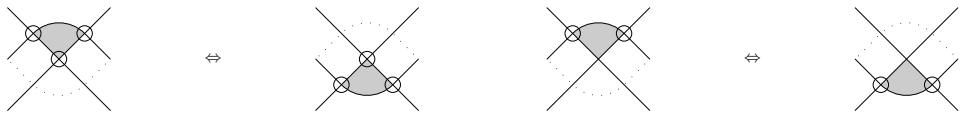

    \centerline{\epsfig{file=doodlefigVR_3.eps, height=1.4cm}
    \qquad \qquad 
    \epsfig{file=doodlefigVFR_3.eps, height=1.4cm}
    }
    \vspace*{8pt}
    \caption{Moves $VR_3$ (Left) and  $VR_4$ (Right)}\label{figVR_3figVR_4}
    \end{figure}

For a virtual diagram $D$, let ${\rm rev}(D)$ or $-D$ denote the virtual diagram obtained from $D$ by reversing the orientation as a $1$-manifold.

We denote by ${\bf Diagram}_{\rm strict}(n)$ the set of virtual diagrams with $n$ real crossings modulo strict equivalence, and by  
 ${\bf Diagram}_{\rm strict+rev}(n)$ the set of virtual diagrams with $n$ real crossings modulo strict equivalence and reversing orientations.

A virtual diagram 
can be described by a sequence  called a {\it Gauss code} (see Section~\ref{sect:Gauss}).  
Such a sequence is not unique. We introduce the notion of a {\it left canonical Gauss code}.  

\begin{theorem} \label{thm:DGL}
There are bijections 
\begin{equation} \label{eqn:DGL}
 {\bf Diagram}_{\rm strict}(n)  \longleftrightarrow {\bf Gauss}^{\rm LC}(n) 
\end{equation}
and 
\begin{equation} \label{eqn:DrGL}
{\bf Diagram}_{\rm strict+rev}(n) \longleftrightarrow {\bf Gauss}^{\rm LC}_{\rm rev}(n),  
\end{equation}
where ${\bf Gauss}^{\rm LC}(n)$ is the set of left canonical Gauss codes with $n$ letters. 
The set ${\bf Gauss}^{\rm LC}_{\rm rev}(n)$ is explained later.  
\end{theorem} 

\vspace{0.3cm}

We refer to the moves depicted in Figure~\ref{figFR_1figFR_2} as $R_1$ and $R_2$, which are flat versions of Reidemeister moves of type $1$ and type $2$.

    \begin{figure}[h]
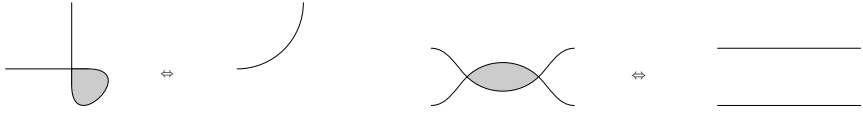

    \centerline{\epsfig{file=doodlefigFR_1.eps, height=1.4cm}
    \qquad \qquad 
    \epsfig{file=doodlefigFR_2.eps, height=0.8cm}
    }
    \vspace*{8pt}
    \caption{Moves $R_1$ (Left) and $R_2$ (Right)}\label{figFR_1figFR_2}
    \end{figure}

Two virtual diagrams are {\it orientedly doodle-equivalent}  if they are related by a finite sequence of 
moves $R_1$, $R_2$, $VR_1$, $VR_2$, $VR_3$, $VR_4$ and isotopies of $\R^2$, or equivalently if they are related by a finite sequence of 
moves $R_1$, $R_2$, detour moves and isotopies of $\R^2$.  
Two virtual diagrams $D$ and $D'$ are {\it unorintedly doodle-equivalent} if $D$ is orientedly equivalent to $D'$ or ${\rm rev}(D')$, i.e., they are related by a finite sequence of $R_1$, $R_2$, $VR_1$, $VR_2$, $VR_3$, $VR_4$ and  
isotopies of $\R^2$ and reversing orientations. 

\begin{definition}[\cite{BFKK}]{\rm 
An {\it oriented} (or {\it unoriented}) {\it virtual doodle} is an oriented (or unoriented) doodle-equivalence class of virtual diagrams. 
}\end{definition}



When virtual crossings are not allowed, virtual doodles are doodles in the original sense 
defined by the second author and P. Taylor \cite{FT} and by M. Khovanov \cite{MK}. 

It is proved in \cite{BFKK}  that there is a natural bijection between 
the set of oriented (or unoriented) virtual doodles and the set of 
oriented (or unoriented) doodles on surfaces. This fact is analogous to the fact that 
there is a natural bijection between 
the set of oriented (or unoriented) virtual knots and the set of stably equivalence classes of 
oriented (or unoriented) knot diagrams on surfaces (cf. \cite{CKS, KK}).

A virtual diagram is {\it minimal} if we cannot apply any $R_1$ or $R_2$ move even 
after changing diagrams up to strict equivalence.  

\begin{theorem}[\cite{BFKK}]\label{thm0}
An oriented (or unoriented) virtual doodle has a unique minimal representative. 
That is, any oriented (or unoriented) virtual doodle can be represented by a minimal virtual diagram and such a diagram is unique up to strict  equivalence (or up to strict  equivalence and reversing orientations).  
\end{theorem}

This is analogous to Kuperberg's theorem on virtual knots \cite{GK}.  

Theorem~\ref{thm0} states that there are bijections 
\begin{equation}
 {\bf Doodle}_{\rm ori} \longleftrightarrow {\bf Diagram}^{\rm min}_{\rm strict}
 \end{equation}
and 
\begin{equation}
  {\bf Doodle}_{\rm unori} \longleftrightarrow {\bf Diagram}^{\rm min}_{\rm strict+rev}, 
  \end{equation}
where 
${\bf Doodle}_{\rm ori}$ (or ${\bf Doodle}_{\rm unori}$) is the set of 
oriented  (or unoriented)  virtual doodles and 
${\bf Diagram}^{\rm min}_{\rm strict}$ (or ${\bf Diagram}^{\rm min}_{\rm strict+rev}$) is 
the set of minimal virtual diagrams modulo strict equivalence (or modulo strict equivalence and reversing  orientations).

The {\it minimum real crossing number} of an oriented (or unoriented) virtual doodle is the minimum number among the numbers of real crossings of all virtual diagrams representing the oriented (or unoriented) virtual doodle.  

Let ${\bf Doodle}_{\rm ori}(n)$ (or ${\bf Doodle}_{\rm unori}(n)$) be the restriction of  
${\bf Doodle}_{\rm ori}$ (or ${\bf Doodle}_{\rm unori}$) to those 
with minimum real crossing number $n$, and let 
${\bf Diagram}^{\rm min}_{\rm strict}(n)$ (or ${\bf Diagram}^{\rm min}_{\rm strict+rev}(n)$) 
be the restriction of 
${\bf Diagram}^{\rm min}_{\rm strict}$ (or ${\bf Diagram}^{\rm min}_{\rm strict+rev}$) 
to those with real crossing number $n$.

Then Theorem~\ref{thm0} implies that there are bijections 
\begin{equation} \label{eqn:DDori}
{\bf Doodle}_{\rm ori}(n) \longleftrightarrow 
{\bf Diagram}^{\rm min}_{\rm strict}(n)
\end{equation}
and 
\begin{equation} \label{eqn:DDunori}
{\bf Doodle}_{\rm unori}(n) \longleftrightarrow
{\bf Diagram}^{\rm min}_{\rm strict+rev}(n).  
\end{equation}

On the other hand, the bijections in (\ref{eqn:DGL}) and (\ref{eqn:DrGL}) implies 

\begin{theorem} \label{thm:DGLmin}
There are bijections 
\begin{equation} \label{eqn:DGori}
{\bf Diagram}^{\rm min}_{\rm strict}(n) 
\longleftrightarrow {\bf Gauss}^{\rm min, LC}(n) 
\end{equation}
and 
\begin{equation} \label{eqn:DGunori}
{\bf Diagram}^{\rm min}_{\rm strict+rev}(n) 
\longleftrightarrow   {\bf Gauss}^{\rm min, LC}_{\rm rev}(n),   
\end{equation}
where ${\bf Gauss}^{\rm min, LC}(n)$ is the set of minimal and left canonical  Gauss codes with $n$ letters. 
The set ${\bf Gauss}^{\rm min, LC}_{\rm rev}(n)$ is explained later.  
\end{theorem}

Combining these, we have 

\begin{theorem} \label{thm:DDG}
There are bijections 
\begin{equation} \label{eqn:DDGori}
{\bf Doodle}_{\rm ori}(n) \longleftrightarrow 
{\bf Diagram}^{\rm min}_{\rm strict}(n) 
\longleftrightarrow {\bf Gauss}^{\rm min, LC}(n) 
\end{equation}
and 
\begin{equation} \label{eqn:DDGunori}
{\bf Doodle}_{\rm unori}(n) \longleftrightarrow
{\bf Diagram}^{\rm min}_{\rm strict+rev}(n) 
\longleftrightarrow   {\bf Gauss}^{\rm min, LC}_{\rm rev}(n).  
\end{equation}

\end{theorem}

This paper is organized as follows: 
We introduce Gauss codes for virtual diagrams in Section~\ref{sect:Gauss}, 
left preferred Gauss codes in Section~\ref{section:leftpreferred} and 
left canonical Gauss codes in Section~\ref{sect:leftcanonical}. 
Theorem~\ref{thm:DGL} is proved in Section~\ref{sect:leftcanonical}. 
In Section~\ref{section:minimal} we discuss minimal Gauss codes and 
prove Theorem~\ref{thm:DGLmin}.  Section~\ref{sect:list} is devoted to demonstration of making a list of virtual doodles using Theorem~\ref{thm:DDG}.  
In Section~\ref{codes_orientations} we summarize the results we have seen for virtual diagrams and virtual doodles, and  
then we introduce canonical orientations for unoriented virtual diagrams and unoriented virtual doodles.  
In Section~\ref{section:arrow} arrow diagrams for virtual doodles are discussed.

\section{Gauss codes} \label{sect:Gauss}

A {\it Gauss code} on $n$ letters is a sequence $x_1 x_2 \dots x_{2n}$ 
of length $2n$ such that each $x_i$ is an element of
$$J =\{(1,L), (1,R), (2, L), (2,R), .., (n,L), (n,R)\}$$  
and all elements of $J$ appear in the sequence.

\vspace{0.3cm} 
Let $D$ be a virtual diagram. Let $c$ be a real crossing of $D$ and let $N(c)$ be a regular neighbourhood of $c$. 
The intersection of $D$ and $N(c)$ is a pair of two oriented short arcs intersecting transversely at $c$. 
We call the arcs the {\it branches} at $c$.  

\begin{itemize}
\item{}
As we move through a real crossing along $D$ on a branch $b$, if the other branch of the real crossing passes from the left to the right (or right to the left), then the branch $b$ is called a {\it left} branch (or a {\it right} branch). 
\end{itemize}

Let $D$ be a virtual diagram with $n$ $(>0)$ real crossings. It is decomposed into $2n$ branches and $2n$ semiarcs. 
We denote by $X(D)$, $B(D)$ and $A(D)$ the set of real crossings, the branches and the semiarcs of D, respectively.  

A {\it labeling of real crossings} of $D$ is a bijection $f: X(D) \to \{1, \dots, n\}$ from 
the real crossings to integers $1, \dots, n$.   

For a labeling of real crossings $f: X(D) \to \{1, \dots, n\}$, there is a unique 
bijection $F : B(D) \to J$  such that for each branch $b$, $F(b)$ 
is $(j, L)$ or $(j, R)$ where 
 $j$ is $f$ of the real crossing on the branch $b$ and the second component 
means that the branch is a left $(L)$ or right $(R)$ branch.   
We call the map $F : B(D) \to J$ the  {\it $J$-labeling of branches} of $D$ associated with $f$ and 
the value $F(b)$ the {\it $J$-label} of $b$ associated with $f$.

Take a semiarc $a_0$ of $D$, which we call a {\it base semiarc}.  

Fix a base semiarc $a_0$ of $D$ and a labeling $f$ of real crossings.  
Let $b_1, \dots, b_{2n}$ be branches at real crossings of $D$ appearing in this order as we move along $D$ from $a_0$.  

The {\it Gauss code} of $D$ is defined as a sequence $w= x_1 x_2 \dots x_{2n}$ 
such that for each $i \in \{1, \dots, 2n\}$, $x_i$ is the $J$-label of the $i$-th branch $b_i$ associated with $f$, i.e.,  
it is 
$(j, L)$ or $(j, R)$ where 
$j$ is $f$ of the real crossing on $b_i$, and 
$L$ or $R$ in the second component means that  $b_i$ is a left or right branch.  

It depends on the base semiarc $a_0$ and the labeling $f$ of  real crossings. 
When we do not specify the base semiarc or the labeling of real crossings, it is just called a Gauss code of $D$.    

If two virtual diagrams $D$ and $D'$ equipped a base semiarc and a labeling of real crossings 
are related by detour moves with respect to the base semiarc and the labeling of real crossings, then their
Gauss codes are the same. Thus the Gauss code is also defined for a strict equivalence class 
of virtual diagrams equipped with a base semiarc and a labeling of real crossings. 

\begin{proposition}\label{propAA}
There is a bijection from the set of strict equivalence classes of virtual diagrams with $n$ real crossings 
equipped with a base semiarc and a labeling of real crossings to the set of 
Gauss codes on $n$ letters. 
\end{proposition}

\begin{proof} 
For a Gauss code on $n$ letters, 
a virtual diagram equipped with a base semiarc and a labeling of real crossings  is constructed as follows: 
Let $w= x_1 x_2 \dots x_{2n}$ be a Gauss code. 
Consider $n$ mutually disjoint discs $N_1, \dots, N_n$ on $\R^2$ and 
consider a pair of oriented, properly embedded arcs in each disc $N_j$, $j=1,\dots, n$, intersecting transversely 
at a point.  We assume that the pair of arcs have $J$-labels $(j, L)$ and $(j, R)$. 
Connect these arcs in  $N_1, \dots, N_n$ by using $2n$ immersed arcs outside of $N_1, \dots, N_n$ 
to obtain a virtual diagram with the Gauss code $w= x_1 x_2 \dots x_{2n}$.  (All crossings of the immersed arcs outside of $N_1, \dots, N_n$ are considered as virtual crossings.)  

Suppose that two diagrams $D$ and $D'$ equipped with a base semiarc and a labeling of real crossings have the same Gauss code.  By an isotopy of $\R^2$ we assume that  they are identical near the real crossings with the same label.  The difference of $D$ and $D'$ are semiarcs, and the connection of branches by the semiarcs of $D$ is the same with that of $D'$.  Thus $D$ and $D'$ are strictly equivalent with respect to the base point and the labeling of real crossings. 
\end{proof}

We have considered the Gauss code for a virtual diagram equipped with a base semiarc and a labeling of real crossings. When we change the base semiarc and the labeling of real crossings, 
the Gauss code is transformed by a simple rule. Let us observe this. 

Let ${\bf Gauss}(n)$ be the set of Gauss codes on $n$ letters. 

For a permutation $\pi$ of $\{1,2, \dots, n\}$, define a map 
\begin{equation}
\pi_\ast : {\bf Gauss}(n) \to {\bf Gauss}(n)
\end{equation}
as follows: 
For a Gauss code  $w$, 
$\pi_\ast(w)$ is the Gauss code such that if the $i$-th element of $w$ is $(j, L)$ (or $(j, R)$)  then  
the $i$-th element is  $(\pi (j), L)$ (or $(\pi (j), R)$) for each $i \in \{1, \dots, 2n\}$.   

For an integer $m$ $(\in \{0,1, \dots, 2n-1\})$, define a map 
\begin{equation}
{\rm shift}[m] : {\bf Gauss}(n) \to {\bf Gauss}(n)
\end{equation}
as follows: 
For a Gauss code  $w=x_1 x_2 \dots x_{2n}$, 
${\rm shift}[m](w)$ is the Gauss code 
$$x_{m+1} x_{m+2} \dots x_{2n} x_1 x_2 \dots x_m. $$  

Define a map 
\begin{equation} 
{\rm rev}: {\bf Gauss}(n) \to {\bf Gauss}(n)
\end{equation} 
as follows: For a Gauss code $w= x_1 x_2 \dots x_{2n}$, let ${\rm rev}(w) = x_{2n} \dots x_2 x_1$. 

The following lemma is obtained directly from the definition and we leave the proof to the reader. 

\begin{lemma}\label{lemmaAA}
Let $D$ be a virtual digram with $n$ real crossings.  
\begin{itemize} 
\item[(1)] 
Fix a base semiarc $a_0$ on $D$. Let $w$ and $w'$ be the Gauss codes of $D$ 
with respect to labelings $f$ and $f'$ of real crossings, respectively.  
Then $w' = (f' \circ f^{-1})_\ast(w)$.  
\item[(2)] 
Fix a labeling of real crossings. 
Let $w$ and $w'$ be the Gauss codes of $D$ 
with respect to base semiarcs $a_0$ and $a_0'$, respectively.  
Let $m$ be the number of branches we meet when we move along $D$ from $a_0$ to $a_0'$.   
Then  
$w' = {\rm shift}[m](w)$.  
\item[(3)]  
Fix a base semiarc $a_0$ on $D$ and a labeling of real crossings. Let 
$w$ and $w'$ be the Gauss codes of $D$ and 
${\rm rev}(D)$. Then  $w' = {\rm rev}(w)$. 
\end{itemize}
\end{lemma}

\begin{definition}\label{definition:equivGauss} {\rm 
Two Gauss codes on $n$ letters $w$ and $w'$ are {\it orientedly equivalent} if 
there is a permutation $\pi$ and an integer $m$ such that $w' = {\rm shift}[m] \circ \pi_\ast(w)$.  
Two Gauss codes on $n$ letters $w$ and $w'$ are {\it unorientedly equivalent} if 
$w$ is orientedly equivalent to $w'$ or ${\rm rev}(w')$. 
}\end{definition}

Let ${\bf Gauss}_{\rm ori}(n)$ (or ${\bf Gauss}_{\rm unori}(n)$)
denote the oriented (or unoriented) equivalence classes of Gauss codes on $n$ letters. 
There are consecutive natural projections 
\begin{equation}
{\bf Gauss}(n) \to 
{\bf Gauss}_{\rm ori}(n) \to {\bf Gauss}_{\rm unori}(n).
\end{equation}

\begin{proposition}\label{propAB}
There are bijections 
\begin{equation}
{\bf Diagram}_{\rm strict}(n) \longleftrightarrow {\bf Gauss}_{\rm ori}(n) 
\quad \mbox{and}\quad 
 {\bf Diagram}_{\rm strict+rev}(n) \longleftrightarrow {\bf Gauss}_{\rm unori}(n). 
\end{equation}

\end{proposition}

\begin{proof} 
This is a consequence of Proposition~\ref{propAA} and Lemma~\ref{lemmaAA}.  
\end{proof}

\section{Left preferred Gauss codes}  \label{section:leftpreferred}

\begin{definition}{\rm  Let  $w = x_1 x_2 \dots x_{2n}$  be a Gauss code.  
 It is {\it weakly left preferred} if  $(1, L) (2, L)\dots (n, L)$ is a subsequence.   
 It is  {\it left preferred} if it is weakly left preferred and $x_1 = (1, L)$.  
 }\end{definition} 
 
 Let ${\bf Gauss}^{\rm LP}(n)$  be the set of left preferred Gauss codes with $n$ letters. 
 
First we consider a map 
 \begin{equation} \label{eqn:projLP}
 {\rm proj^{LP}} : {\bf Gauss}(n) \to {\bf Gauss}^{\rm LP}(n).   
\end{equation}

 Let $w = x_1 x_2 \dots x_{2n}$ be a Gauss code on $n$ letters. 
 Let $(j_1, L)(j_2, L) \dots (j_n, L)$ be the subsequence of $w$ obtained by removing elements whose second component is $R$.  Let $\pi$ be a permutation of $(1,\dots,n)$ which sends $j_1, j_2, \dots, j_n$ to $1,2, \dots, n$, respectively.   
 Then $\pi_\ast(w)$ is weakly left preferred.  
 Let $m$ be an integer in $\{1, \dots, 2n\}$ such that the $m$-th element of $\pi_\ast(w)$ is $(1,L)$. 
 Apply ${\rm shift}[m-1]$ to  $\pi_\ast(w)$, and we have a  left preferred Gauss code.  
 This is the definition of ${\rm proj^{LP}}(w)$.  By definition, $w$ and ${\rm proj^{LP}}(w)$ are orientedly equivalent Gauss codes. 
 
If $w \in {\bf Gauss}^{\rm LP}(n)$ then ${\rm proj^{LP}} (w) =w$.  In particular,  ${\rm proj^{LP}} ({\rm proj^{LP}} (w)) = {\rm proj^{LP}} (w)$ for any 
Gauss code $w$.  

By a direct computation we see that for any Gauss code $w$, 
 \begin{equation}\label{eqn:projrev}
  {\rm proj^{LP}}({\rm rev}(w) )= 
 {\rm proj^{LP}}({\rm rev}( {\rm proj^{LP}}(w))).  
 \end{equation}

\vspace{0.3cm}

Define a map 
 \begin{equation}
 {\rm shift^{LP}} : {\bf Gauss}^{\rm {LP}}(n) \to {\bf Gauss}^{\rm {LP}}(n)  
\end{equation}
by 
\begin{equation}
{\rm shift^{LP}}(w) = {\rm proj^{LP}}({\rm shift}[1](w)). 
 \end{equation}

Let $w = x_1 x_2 \dots x_{2n}$ be a left preferred Gauss code.  Assume $n \geq 2$. 
Let $m$ be an integer with $x_m=(2, L)$ and let $\pi $ be a permutation of $\{1, \dots, n\}$ 
sending $i$ to $i-1$ mod $n$ for each $i \in \{1, \dots, n\}$.  
Then ${\rm shift^{LP}}(w) = {\rm shift}[m-1] \circ \pi_\ast (w)$.

\begin{remark}{\rm 
When $n \geq 2$, 
for each $k=1, \dots, n-1$,  
 \begin{equation}
 ({\rm shift^{LP}})^k : {\bf Gauss}^{\rm {LP}}(n) \to {\bf Gauss}^{\rm {LP}}(n)  
\end{equation}
is computed as follows. Let $w = x_1 x_2 \dots x_{2n}$ be a left preferred Gauss code. 
Let $m$ be an integer with $x_m= (k+1, L)$ and let $\pi$ be a permutation of $\{1, \dots, n\}$ 
sending $i$ to $i-k$ mod $n$ for each $i \in \{1, \dots, n\}$.  
Then $({\rm shift^{LP}})^k(w) = {\rm shift}[m-1] \circ (\pi)_\ast (w)$.  
}\end{remark}

\begin{example}\label{example:shift}{\rm 
Let $w = (1, L) (2, R) (2, L) (3, R) (1,R) (3, L) \in {\bf Gauss}^{\rm {LP}}(3)$. Then 
\begin{equation}
\begin{array}{rcl}
w &=& (1, L) (2, R) (2, L) (3, R) (1,R) (3, L), \\ 
{\rm shift^{LP}}(w) &=&  (1, L) (2, R) (3, R) (2, L) (3, L) (1, R), \\ 
({\rm shift^{LP}})^2(w) &=&   (1, L) (2, L) (3, R) (3, L) (1, R) (2, R).  
\end{array}
\end{equation}
}\end{example}

We define a map 
 \begin{equation}
 {\rm rev^{LP}} : {\bf Gauss}^{\rm {LP}}(n) \to {\bf Gauss}^{\rm {LP}}(n)  
\end{equation}
by 
 \begin{equation}
{\rm rev^{LP}}(w) = {\rm proj^{LP}}({\rm rev}(w)).  
\end{equation}

The maps ${\rm shift^{LP}}$ and ${\rm rev^{LP}}$ are 
bijections of ${\bf Gauss}^{\rm {LP}}(n)$ satisfying that 
\begin{equation}
({\rm shift^{LP}})^n =1 \quad \mbox{and} \quad ({\rm rev^{LP}})^2 = 1, 
\end{equation}
where $1$ is the identity map.

\begin{definition}{\rm 
Two left preferred Gauss codes on $n$ letters are 
{\it orientedly equivalent as left preferred Gauss codes} if 
they are in the same orbit by the group action generated by ${\rm shift^{LP}}$.   
Two left preferred Gauss codes on $n$ letters are  
 {\it unorientedly equivalent as left preferred Gauss codes} if they are 
in the same orbit by the group action generated by ${\rm shift^{LP}}$ and 
${\rm rev^{LP}}$.  
}\end{definition}

\begin{theorem}\label{theorem:leftA}
Let $w$ and $w'$ be Gauss codes on $n$ letters.  
They are orientedly (or unorientedly) equivalent as Gauss codes 
(Definition~\ref{definition:equivGauss}) 
if and only if 
${\rm proj^{LP}}(w)$ and ${\rm proj^{LP}}(w')$ are orientedly 
(or unorientedly) equivalent as  left preferred Gauss codes. 
\end{theorem}

\begin{proof} 
First we show the only if part.  

(1) Suppose $w' = \pi_\ast (w)$ for some permutation $\pi$.  
It is obvious that ${\rm proj^{LP}}(w') = {\rm proj^{LP}}(w)$.  

(2) Suppose that $w' =  {\rm shift}[1](w)$.  Let $x_1$ be the first element of $w$. 
If $x_1 =(j,R)$ for some $j$, then 
${\rm proj^{LP}}(w') = {\rm proj^{LP}}(w)$. If $x_1 =(j,L)$ for some $j$, then 
${\rm proj^{LP}}(w') = {\rm shift^{LP}}({\rm proj^{LP}}(w))$.  

(3) Suppose $w' =  {\rm rev}(w)$.  By (\ref{eqn:projrev}), we have 
$${\rm proj^{LP}}(w') = {\rm proj^{LP}}({\rm rev}(w)) = {\rm proj^{LP}}({\rm rev}({\rm proj^{LP}}(w))) =
{\rm rev^{LP}}({\rm proj^{LP}}(w)).$$  

Therefore, we have the only if part. 

We show the if part.   Suppose that ${\rm proj^{LP}}(w)$ and ${\rm proj^{LP}}(w')$ are orientedly 
(or unorientedly) equivalent as  left preferred Gauss codes.  Then 
they are orientedly 
(or unorientedly) equivalent as Gauss codes. 
Note that $w$ and ${\rm proj^{LP}}(w)$ are orientedly equivalent as Gauss codes, and so are 
$w'$ and ${\rm proj^{LP}}(w')$.  Thus $w$ and $w$ are orientedly 
(or unorientedly) equivalent as Gauss codes. 
\end{proof} 

\begin{corollary}
Two left preferred Gauss codes are orientedly (or unorientedly) equivalent as Gauss codes 
if and only if they are orientedly (or unorientedly) equivalent as left preferred Gauss codes.  
\end{corollary}

Let ${\bf Gauss}^{\rm {LP}}_{\rm ori}(n)$
(or ${\bf Gauss}^{\rm {LP}}_{\rm unori}(n)$)  be the set of 
oriented (or unoriented) equivalence classes of  left preferred
Gauss codes on $n$ letters.   

\begin{theorem}\label{theorem:leftB} 
The projection map ${\rm proj^{LP}}: 
{\bf Gauss}(n) \to {\bf Gauss}^{\rm {LP}}(n) $
 induces bijections: 
\begin{equation}
{\bf Gauss}_{\rm ori}(n) \longleftrightarrow {\bf Gauss}^{\rm {LP}}_{\rm ori}(n) 
\quad \mbox{and} \quad 
{\bf Gauss}_{\rm unori}(n) \longleftrightarrow {\bf Gauss}^{\rm {LP}}_{\rm unori}(n).  
\end{equation}
\end{theorem}

\begin{proof}
It follows from Theorem~\ref{theorem:leftA}. 
\end{proof}

\section{Left canonical Gauss codes}\label{sect:leftcanonical}

We fix an order on $J$ with 
\begin{equation}
 (1, L) < (1, R) < (2, L) < (2, R) < \dots < (n, L) < (n, R),  
 \end{equation}
and assume that ${\bf Gauss}(n)$ and its subset ${\bf Gauss}^{\rm {LP}}(n)$ are ordered sets with 
a lexicographical order using the order of $J$.  

For a left preferred Gauss code $w$ with $n$ letters,  
the oriented equivalence class  of $w$ as 
left preferred Gauss codes, 
 $[w]^{\rm LP}_{\rm ori} \in {\bf Gauss}^{\rm {LP}}_{\rm ori}(n)$, 
 is given by 
\begin{equation}
[w]^{\rm LP}_{\rm ori} = \{
w,~ {\rm shift^{LP}}(w),~ ({\rm shift^{LP}})^2(w),~  \dots,~ ({\rm shift^{LP}})^{n-1}(w)
\}.
\end{equation}

\begin{definition}{\rm 
A {\it left canonical Gauss code} is a left preferred Gauss code $w$ such that 
it is the smallest element in $[w]^{\rm LP}_{\rm ori}$,  
the oriented equivalence class as left preferred Gauss codes.   
The {\it left canonical representative} of an oriented equivalence class 
of left preferred Gauss codes is the smallest representative.  
}\end{definition}

For example, 
let $w = (1, L) (2, R) (2, L) (3, R) (1,R) (3, L)$ (as in Example~\ref{example:shift}). 
Then 
\begin{equation}
\begin{array}{rcl}
w &=& (1, L) (2, R) (2, L) (3, R) (1,R) (3, L), \\ 
{\rm shift^{LP}}(w) &=&  (1, L) (2, R) (3, R) (2, L) (3, L) (1, R), \\ 
({\rm shift^{LP}})^2(w) &=&   (1, L) (2, L) (3, R) (3, L) (1, R) (2, R).  
\end{array}
\end{equation}
The smallest element is $(1, L) (2, L) (3, R) (3, L) (1, R) (2, R)$. 
It is the left canonical representative of $[w]^{\rm LP}_{\rm ori}$ and   
it is a left canonical Gauss code. The other two Gauss codes $w$ and ${\rm shift^{LP}}(w)$ are not left canonical Gauss codes.  

Let ${\bf Gauss}^{\rm LC}(n)$ be the set of left canonical Gauss codes of $n$ letters. 
Sending oriented equivalence classes  
of left preferred Gauss codes to their left canonical representatives, we have 
a bijection 
\begin{equation}\label{eqn:GaussoriLPLC}
{\bf Gauss}^{\rm {LP}}_{\rm ori}(n) \longleftrightarrow {\bf Gauss}^{\rm LC}(n). 
\end{equation}

We define a map 
\begin{equation}
{\rm proj^{LC}} : {\bf Gauss}(n) \to {\bf Gauss}^{\rm LC}(n)
\end{equation}
by sending a Gauss code $w$ to the left canonical representative of 
the oriented equivalence class $[{\rm proj^{LP}}(w)]^{\rm LP}_{\rm ori}$ 
of ${\rm proj^{LP}}(w)$.  

For example, let $w = (1, L) (2, R) (2, L) (3, R) (1,R) (3, L)$ as above. Then 
${\rm proj^{LC}}(w)= (1, L) (2, L) (3, R) (3, L) (1, R) (2, R)$.  

If $w \in {\bf Gauss}^{\rm LC}(n)$ then ${\rm proj^{LC}}(w) =w$.  
In particular, for any Gauss code $w$, ${\rm proj^{LC}}({\rm proj^{LC}}(w)) ={\rm proj^{LC}}(w)$. 

By Proposition~\ref{propAB}, Theorem~\ref{theorem:leftB} and the bijection in (\ref{eqn:GaussoriLPLC}), we 
obtain a bijection 
\begin{equation}
 {\bf Diagram}_{\rm strict}(n)  \longleftrightarrow {\bf Gauss}^{\rm LC}(n),  
\end{equation}
which is claimed in Theorem~\ref{thm:DGL}

We define a map 
\begin{equation}
{\rm rev^{LC}} : {\bf Gauss}^{\rm LC}(n) \to {\bf Gauss}^{\rm LC}(n)
\end{equation}
by 
\begin{equation}
{\rm rev^{LC}} (w) =  {\rm proj^{LC}}( {\rm rev}(w) ). 
\end{equation}

Let ${\bf Gauss}^{\rm LC}_{\rm rev}(n)$ be the quotient of 
${\bf Gauss}^{\rm LC}(n)$ by the action of ${\rm rev^{LC}}$, i.e., the elements of 
${\bf Gauss}^{\rm LC}_{\rm rev}(n)$ are $\{w, {\rm rev^{LC}}(w) \}$ for  $w \in {\bf Gauss}^{\rm LC}(n)$.

Let $w$ be a left preferred Gauss code with $n$ letters. 
The unoriented equivalence class  of $w$ as 
left preferred Gauss codes, 
 $[w]^{\rm LP}_{\rm unori} \in {\bf Gauss}^{\rm {LP}}_{\rm unori}(n)$, 
 is given by 
\begin{equation}
\begin{array}{ccl}
[w]^{\rm LP}_{\rm unori}   
  & = & [w]^{\rm LP}_{\rm ori} \cup [w']^{\rm LP}_{\rm ori}  \\
  & = &  
\{
w,~ {\rm shift^{LP}}(w),~ ({\rm shift^{LP}})^2(w),~  \dots,~ ({\rm shift^{LP}})^{n-1}(w)
\}
 \\
  &   &   \cup 
\{
w',~ {\rm shift^{LP}}(w'),~ ({\rm shift^{LP}})^2(w'),~  \dots,~ ({\rm shift^{LP}})^{n-1}(w')
\}, 
\end{array}
\end{equation}
where $w' = {\rm rev^{LP}}(w) = {\rm proj^{LP}}( {\rm rev}(w) )$.  
Then $\{   {\rm proj^{LC}}(w),       {\rm proj^{LC}}( {\rm rev}(w)) \}$ is an element of 
${\bf Gauss}^{\rm LC}_{\rm rev}(n)$.  

We have 
a bijection 
\begin{equation}\label{eqn:GaussunoriLPLC}
{\bf Gauss}^{\rm {LP}}_{\rm unori}(n) \longleftrightarrow {\bf Gauss}^{\rm LC}_{\rm rev}(n). 
\end{equation}
by sending 
$[w]^{\rm LP}_{\rm unori} \in {\bf Gauss}^{\rm {LP}}_{\rm unori}(n)$ 
to 
$\{   {\rm proj^{LC}}(w),       {\rm proj^{LC}}( {\rm rev}(w)) \}$  
for any $w \in {\bf Gauss}^{\rm {LP}}(n)$.  

By Proposition~\ref{propAB}, Theorem~\ref{theorem:leftB} and the bijection in (\ref{eqn:GaussunoriLPLC}), we 
obtain a bijection  
\begin{equation}
 {\bf Diagram}_{\rm strict+rev}(n)  \longleftrightarrow {\bf Gauss}^{\rm LC}_{\rm rev}(n),   
\end{equation}
which is claimed in Theorem~\ref{thm:DGL}.

\begin{definition}{\rm 
The {\it left canonical representative} of an unoriented equivalence class 
of left preferred Gauss codes is the smallest representative.  
}\end{definition}

In other words, for any left preferred Gauss code $w$, 
the left canonical representative of $[w]^{\rm LP}_{\rm unori}$ is the smaller one between the left canonical representative  ${\rm proj^{LC}}(w)$ 
of $[w]^{\rm LP}_{\rm ori}$ and 
the left canonical representative  ${\rm proj^{LC}}( {\rm rev}(w))$ of 
$[{\rm rev^{LP}}(w)]^{\rm LP}_{\rm ori}$.

\section{Minimal Gauss codes} \label{section:minimal}

We discuss minimal Gauss codes, which are Gauss codes of minimal virtual diagrams representing virtual doodles.  

For a Gauss code $w= x_1 x_2 \dots x_{2n}$, we assume that $x_{2n+1}$ is $x_1$.   

A Gauss code $w= x_1 x_2 \dots x_{2n}$ is {\it $1$-reducible} if there 
is an integer $i \in \{1, \dots, 2n\}$ such that 
$$(x_i, x_{i+1}) = ( (j, L), (j, R) )  \quad \mbox{or} \quad ( (j, R), (j, L) ), $$  
for some $j$.  Otherwise, it is called  {\it $1$-irreducible}. 

A Gauss code $w= x_1 x_2 \dots x_{2n}$ is {\it $2$-reducible} if one of the following holds: 
\begin{itemize}
\item[(1)] There are integers $i, i' \in \{1, \dots, 2n\}$ such that 
  \begin{equation}
  (x_i, x_{i+1}) = ( (j, L), (k, R) )  \quad \mbox{and} \quad (x_{i'}, x_{i'+1}) = ( (j, R), (k, L) ) 
  \end{equation}
for some $j \neq k$. 
\item[(2)]
There are integers $i, i' \in \{1, \dots, 2n\}$ such that 
  \begin{equation}
  (x_i, x_{i+1}) = ( (j, L), (k, R) )  \quad \mbox{and} \quad (x_{i'}, x_{i'+1}) = ( (k, L), (j, R) ) 
  \end{equation}
or 
  \begin{equation}
  (x_i, x_{i+1}) = ( (j, R), (k, L) )  \quad \mbox{and} \quad (x_{i'}, x_{i'+1}) = ( (k, R), (j, L) ) 
  \end{equation} 
for some $j \neq k$. 
\end{itemize}
Otherwise, it is called  {\it $2$-irreducible}.  

We say that $w$ is {\it minimal} or 
{\it irreducible} if it is $1$-irreducible and $2$-irreducible. 

\begin{lemma}\label{lemmaC}
Let $D$ be a virtual diagram and $w$ a Gauss code of $D$. 
Then $D$ is minimal if and only if $w$ is minimal. 
\end{lemma}

\begin{proof}
It is easily verified that a move $R_1$ can be applied to $D$ (after applying detour moves if necessary) 
if and only if $w$ is $1$-reducible and that a  move $R_2$ can be applied to $D$ (after applying detour moves if necessary) 
if and only if $w$ is $2$-reducible.  
\end{proof}

Note that if a Gauss code $w$ is minimal then any Gauss code which is  
orientedly (or unorientedly) equivalent to $w$ is minimal.

Let ${\bf Gauss}_{\rm ori}^{\rm min}(n)$ (or ${\bf Gauss}_{\rm unori}^{\rm min}(n)$) 
denote the set of 
minimal Gauss codes with $n$ letters modulo oriented (or unoriented) equivalence. 

\begin{proposition}\label{propDB}
There are bijections  
\begin{equation}
{\bf Diagram}^{\rm min}_{\rm strict}(n) \longleftrightarrow  {\bf Gauss}^{\rm min}_{\rm ori}(n) 
\quad \mbox{and} \quad 
{\bf Diagram}^{\rm min}_{\rm strict+rev}(n) \longleftrightarrow  {\bf Gauss}^{\rm min}_{\rm unori}(n). 
\end{equation}
\end{proposition}

\begin{proof}
By Proposition~\ref{propAB} and Lemma~\ref{lemmaC}, we have the result.  
\end{proof}

 Let ${\bf Gauss}^{\rm min, LP}(n)$  be the set of 
 minimal and left preferred Gauss codes with $n$ letters. 

 If $w$ is minimal, then so is ${\rm proj^{LP}}(w)$.  We denote by the same symbols for the restriction of the maps  
 ${\rm proj^{LP}}$,  ${\rm shift^{LP}}$ and  ${\rm rev^{LP}}$ 
  to minimal Gauss codes: 
 \begin{equation}
\begin{array}{rll}
 {\rm proj^{LP}} :& {\bf Gauss}^{\rm min}(n) &\to {\bf Gauss}^{\rm min, LP}(n),   \\ 
 {\rm shift^{LP}} :& {\bf Gauss}^{\rm min, LP}(n) &\to {\bf Gauss}^{\rm min, LP}(n),  
 \quad \mbox{and} \quad \\ 
  {\rm rev^{LP}} :& {\bf Gauss}^{\rm min, LP}(n) &\to {\bf Gauss}^{\rm min, LP}(n).   
\end{array}
\end{equation}

Let ${\bf Gauss}^{\rm min, LP}_{\rm ori}(n)$
(or ${\bf Gauss}^{\rm min, LP}_{\rm unori}(n)$)  be the set of 
oriented (or unoriented)  equivalence classes of minimal and left preferred Gauss codes on $n$ letters.  

\begin{theorem}\label{theorem:leftBmin} 
The projection map ${\rm proj^{LP}}: 
{\bf Gauss}^{\rm min}(n) \to {\bf Gauss}^{\rm min, LP}(n) $ induces  
bijections:  
\begin{equation}
{\bf Gauss}^{\rm min}_{\rm ori}(n) \longleftrightarrow {\bf Gauss}^{\rm min, LP}_{\rm ori}(n) 
\quad \mbox{and} \quad 
{\bf Gauss}^{\rm min}_{\rm unori}(n) \longleftrightarrow {\bf Gauss}^{\rm min, LP}_{\rm unori}(n). 
\end{equation}
\end{theorem}

\begin{proof}
It follows from Theorem~\ref{theorem:leftA} (cf. Theorem~\ref{theorem:leftB}). 
\end{proof}

Let ${\bf Gauss}^{\rm min, LC}(n)$ be the set of minimal and left canonical Gauss codes of $n$ letters. 

If $w$ is minimal, then so is ${\rm proj^{LC}}(w)$.  We denote by the same symbols for the restriction of the maps  
 ${\rm proj^{LC}}$ and  ${\rm rev^{LC}}$ 
  to minimal Gauss codes: 
 \begin{equation}
\begin{array}{rll}
 {\rm proj^{LC}} :& {\bf Gauss}^{\rm min}(n) &\to {\bf Gauss}^{\rm min, LC}(n)  
 \quad \mbox{and} \quad \\ 
  {\rm rev^{LC}} :& {\bf Gauss}^{\rm min, LC}(n) &\to {\bf Gauss}^{\rm min, LC}(n).   
\end{array}
\end{equation}

Let ${\bf Gauss}^{\rm min, LC}_{\rm rev}(n)$ be the quotient of 
${\bf Gauss}^{\rm min, LC}(n)$ by the action of ${\rm rev^{LC}}$, i.e., the elements of 
${\bf Gauss}^{\rm min, LC}_{\rm rev}(n)$ are $\{w, {\rm rev^{LC}}(w) \}$ for  $w \in {\bf Gauss}^{\rm min, LC}(n)$.  Note that ${\bf Gauss}^{\rm min, LC}_{\rm rev}(n)$ is the subset of 
${\bf Gauss}^{\rm LC}_{\rm rev}(n)$ consisting of elements  $\{w, {\rm rev^{LC}}(w) \}$ for  $w \in {\bf Gauss}^{\rm min, LC}(n)$. 

Considering the restrictions of the bijections in Theorem~\ref{thm:DGL}, we have 
bijections  
\begin{equation*}
 {\bf Diagram}^{\rm min}_{\rm strict}(n)  \longleftrightarrow {\bf Gauss}^{\rm min, LC}(n) 
\end{equation*}
and 
\begin{equation*}
 {\bf Diagram}^{\rm min}_{\rm strict+rev}(n)  \longleftrightarrow {\bf Gauss}^{\rm min, LC}_{\rm rev}(n),  
\end{equation*}
which are claimed in Theorem~\ref{thm:DGLmin}.

\section{Making a list of virtual doodles}  \label{sect:list}

By Theorem~\ref{thm:DDG} or Theorem~\ref{theorem:leftBmin}, 
in order to make a list of ${\bf Doodle}_{\rm ori}(n)$,  
we may use  ${\bf Gauss}^{\rm min, LC}(n)$ 
or ${\bf Gauss}^{\rm min, LP}_{\rm ori}(n)$.  

Here is a way to make a list of ${\bf Gauss}^{\rm min, LC}(n)$ or ${\bf Gauss}^{\rm min, LP}_{\rm ori}(n)$.  
\begin{itemize}
\item Make a list of the elements of ${\bf Gauss}^{\rm {LP}}(n)$.  
Removing elements which are 
$1$-reducible or $2$-reducible, we have a list of 
${\bf Gauss}^{\rm min, LP}(n)$.  
\item Let $G[1] = {\bf Gauss}^{\rm min, LP}(n)$. Take the smallest element, say $w_1$, of $G[1]$. 
Make a list of all elements that are orientedly equivalent to $w_1$, which are 
$$w_1, {\rm shift^{LP}}(w_1), ({\rm shift^{LP}})^2(w_1),  \dots, ({\rm shift^{LP}})^{n-1}(w_1).  $$ 
Remove them from $G[1]$ and let $G[2]$ be the result.  
\item Inductively, assume that $G[k]$ is already defined.  If $G[k]$ is non-empty, take the smallest element, say $w_k$, of $G[k]$. Make a list of all elements that are orientedly equivalent to $w_k$.  Remove them from $G[k]$ and let $G[k+1]$ be the result.   
\item This procedure must finish when we have $G[m+1] =\emptyset$ for some $m$. 
Then $\{w_1, \dots, w_m\}$ is a  list of representatives of all elements of ${\bf Gauss}^{\rm min, LP}_{\rm ori}(n)$. Moreover,  $\{w_1, \dots, w_m\}$ is a  list of the elements of ${\bf Gauss}^{\rm min, LC}(n)$. 
\end{itemize}

Once we have a list of ${\bf Gauss}^{\rm min, LC}(n)$ or ${\bf Gauss}^{\rm min, LP}_{\rm ori}(n)$, using ${\rm rev}^{\rm LC}$ or ${\rm rev}^{\rm LP}$, 
 we have 
a list of ${\bf Gauss}^{\rm min, LC}_{\rm rev}(n)$ or ${\bf Gauss}^{\rm min, LP}_{\rm unori}(n)$.

\begin{example}{\rm 
Let us consider a case of $n=3$.  
The order on $J$ is 
\begin{equation}
 (1, L) < (1, R) < (2, L) < (2, R) <  (3, L) < (3, R).  
 \end{equation}
For simplicity, we put 
\begin{equation}\label{eqn:6J}
 1= (1, L) , \quad 2= (1, R),  \quad 3= (2, L), \quad 4= (2, R), \quad 5=  (3, L), \quad 6=  (3, R). 
 \end{equation}
 For example, $w=  (1, L) (2, R) (2, L)  (3, L) (1, R) (3, R)$ is denoted by 
 $(1, 4, 3, 5, 2, 6)$.  
 In this notation, $w=(m_1, m_2, \dots, m_6)$ is left preferred if and only if 
 its subsequence consisting of odd numbers is $(1,3,5)$ and $m_1 =1$.  
 
 Let ${\bf GLP}$ $(= {\bf Gauss}^{\rm {LP}}(3))$ be the set of elements $(m_1, m_2, \dots, m_6)$ which are left preferred.  It consists of $60$ elements: 
\begin{eqnarray*}
&& (1, 2, 3, 4, 5, 6), (1, 2, 3, 4, 6, 5), (1, 2, 3, 5, 4, 6), (1, 2, 3, 5, 6, 4), (1, 2, 3, 6, 4, 5), \\ 
&& (1, 2, 3, 6, 5, 4), (1, 2, 4, 3, 5, 6), (1, 2, 4, 3, 6, 5), (1, 2, 4, 6, 3, 5), (1, 2, 6, 3, 4, 5), \\
&& (1, 2, 6, 3, 5, 4), (1, 2, 6, 4, 3, 5), (1, 3, 2, 4, 5, 6), (1, 3, 2, 4, 6, 5), (1, 3, 2, 5, 4, 6), \\
&& (1, 3, 2, 5, 6, 4), (1, 3, 2, 6, 4, 5), (1, 3, 2, 6, 5, 4), (1, 3, 4, 2, 5, 6), (1, 3, 4, 2, 6, 5), \\
&& (1, 3, 4, 5, 2, 6), (1, 3, 4, 5, 6, 2), (1, 3, 4, 6, 2, 5), (1, 3, 4, 6, 5, 2), (1, 3, 5, 2, 4, 6), \\ 
&& (1, 3, 5, 2, 6, 4), (1, 3, 5, 4, 2, 6), (1, 3, 5, 4, 6, 2), (1, 3, 5, 6, 2, 4), (1, 3, 5, 6, 4, 2), \\
&& (1, 3, 6, 2, 4, 5), (1, 3, 6, 2, 5, 4), (1, 3, 6, 4, 2, 5), (1, 3, 6, 4, 5, 2), (1, 3, 6, 5, 2, 4), \\ 
&& (1, 3, 6, 5, 4, 2), (1, 4, 2, 3, 5, 6), (1, 4, 2, 3, 6, 5), (1, 4, 2, 6, 3, 5), (1, 4, 3, 2, 5, 6), \\
&& (1, 4, 3, 2, 6, 5), (1, 4, 3, 5, 2, 6), (1, 4, 3, 5, 6, 2), (1, 4, 3, 6, 2, 5), (1, 4, 3, 6, 5, 2), \\
&& (1, 4, 6, 2, 3, 5), (1, 4, 6, 3, 2, 5), (1, 4, 6, 3, 5, 2), (1, 6, 2, 3, 4, 5), (1, 6, 2, 3, 5, 4), \\
&& (1, 6, 2, 4, 3, 5), (1, 6, 3, 2, 4, 5), (1, 6, 3, 2, 5, 4), (1, 6, 3, 4, 2, 5), (1, 6, 3, 4, 5, 2), \\
&& (1, 6, 3, 5, 2, 4), (1, 6, 3, 5, 4, 2), (1, 6, 4, 2, 3, 5), (1, 6, 4, 3, 2, 5), (1, 6, 4, 3, 5, 2).
\end{eqnarray*}

In order to make a list of  ${\bf Gauss}^{\rm min, LP}(3)$, we need to remove elements which are 
$1$-reducible or $2$-reducible from ${\bf GLP}$. 

In the notation using (\ref{eqn:6J}), $(m_1, m_2, \dots, m_6)$ is $1$-reducible if and only if 
at least one of 
\begin{equation}
(m_1, m_2), (m_2, m_3), \dots, (m_5, m_6), (m_6, m_1)
 \end{equation} 
belongs to 
\begin{equation}
{\bf A1}:= \{ (1,2), (2,1), (3,4), (4,3), (5,6), (6,5) \}. 
 \end{equation}
 
 Let 
 \begin{equation}
{\bf PrepA2} := \{ (1, 2), (1, 3), (2, 1), (2, 3), (3, 1), (3, 2) \},  
 \end{equation}
which is the set of pairs $(i,j)$ with $i \neq j$ and $1\leq i, j \leq 3$. 
For an element $a \in {\bf PrepA2}$, we denote by $a[1]$ and $a[2]$ 
the first and second components of $a$. 

Let ${\bf A2}$ be a set consisting of all $( (p_1, p_2), (q_1, q_2) )$ with $1 \leq p_1, p_2, q_1, q_2 \leq 6$ such that 
\begin{itemize}
\item[(i)] $( (p_1, p_2), (q_1, q_2) ) = ( (2a[1]-1, 2a[2]), (2a[1], 2a[2]-1)$ 
for some $a \in {\bf PrepA2}$, 
\item[(ii)] $( (p_1, p_2), (q_1, q_2) ) = ( (2a[1]-1, 2a[2]), (2a[2]-1, 2a[1])$ 
for some $a \in {\bf PrepA2}$, or 
\item[(iii)] $( (p_1, p_2), (q_1, q_2) ) = ( (2a[1], 2a[2]-1), (2a[2], 2a[1]-1)$ 
for some $a \in {\bf PrepA2}$.  
\end{itemize}
Then ${\bf A2}$ consists of 
\begin{eqnarray*}
&& ((1, 4), (2, 3)), \quad ((1, 4), (3, 2)), \quad ((1, 6), (2, 5)), \quad ((1, 6), (5, 2)), \quad ((2, 3), (4, 1)), \\ 
&& ((2, 5), (6, 1)), \quad ((3, 2), (1, 4)), \quad ((3, 2), (4, 1)), \quad ((3, 6), (4, 5)), \quad ((3, 6), (5, 4)), \\
&& ((4, 1), (2, 3)), \quad ((4, 5), (6, 3)), \quad ((5, 2), (1, 6)), \quad ((5, 2), (6, 1)), \quad ((5, 4), (3, 6)), \\ 
&& ((5, 4), (6, 3)), \quad ((6, 1), (2, 5)), \quad ((6, 3), (4, 5)). 
 \end{eqnarray*}
 
 In the notation, $(m_1, m_2, \dots, m_6)$ is $2$-reducible if and only if 
 there are integers $i$ and $j$ in $\{1, \dots, 6\}$ such that 
 $( (m_i, m_{i+1}), (m_j, m_{j+1}) )$ belongs to ${\bf A2}$.  
 Here we assume $m_7= m_1$. 
 
 Let ${\bf GLPmin}$ $(={\bf Gauss}^{\rm min, LP}(3))$ be the set of elements 
$(m_1, m_2, \dots, m_6)$ which are minimal and left preferred.  It is 
obtained from ${\bf GLP}$ by removing all elements that are $1$-reducible or $2$-reducible.  
  
 The set ${\bf GLPmin}$ consists of 
 \begin{eqnarray*}
 (1, 3, 2, 6, 4, 5), (1, 3, 5, 2, 6, 4), (1, 3, 5, 4, 2, 6), (1, 3, 6, 4, 2, 5), (1, 4, 2, 6, 3, 5), (1, 6, 4, 2, 3, 5).
 \end{eqnarray*}

This is a complete list of ${\bf Gauss}^{\rm min, LP}(3)$.  
Now we consider oriented equivalence classes.  

For an element $w = (m_1, m_2, \dots, m_6)$ of ${\bf GLPmin}$, 
its oriented equivalence class is 
 \begin{eqnarray*}
 \{ w,  \quad  {\rm shift^{LP}}(w), \quad  ({\rm shift^{LP}})^2(w) \}.  
 \end{eqnarray*}
 
${\rm shift^{LP}}(w)$ is obtained from $(m_1 -2, m_2 -2 , \dots, m_6 -2)$ 
by rotation so that the first element becomes $1$.  Here 
$0$ and $-1$ in $(m_1 -2, m_2 -2 , \dots, m_6 -2)$ 
are identified with $6$ and $5$ respectively. 

$({\rm shift^{LP}})^2(w)$ is obtained from $(m_1 -4, m_2 -4 , \dots, m_6 -4)$ 
by rotation so that the first element becomes $1$.   

For example, the oriented equivalence class of $(1, 3, 2, 6, 4, 5)$ is 
 \begin{eqnarray*}
\{ (1, 3, 2, 6, 4, 5),  \quad (1, 3, 5, 4, 2, 6), \quad (1, 6, 4, 2, 3, 5) \}.  
 \end{eqnarray*}
 The smallest element is $(1, 3, 2, 6, 4, 5)$, which is the left canonical representative of the class. 

Let ${\bf GLPminOri}$ be the set of oriented equivalence classes 
of elements of ${\bf GLPmin}$.   It consists of two elements, 
\begin{equation*}
\begin{array}{llcc}
(1) & d_{3,1}^+ &=& [(1, 3, 2, 6, 4, 5)]_{\rm ori} = \{ (1, 3, 2, 6, 4, 5),  \quad (1, 3, 5, 4, 2, 6), \quad (1, 6, 4, 2, 3, 5) \},  \\ 
(2) & d_{3,1}^- &=& [(1, 3, 5, 2, 6, 4)]_{\rm ori} =  \{ (1, 3, 5, 2, 6, 4),  \quad (1, 3, 6, 4, 2, 5),  \quad (1, 4, 2, 6, 3, 5) \}.  
\end{array}
\end{equation*}
(The meaning of symbols $d_{3,1}^+$ and $d_{3,1}^-$ will be explained later.) 

Note that $(1, 3, 2, 6, 4, 5)$ and $ (1, 3, 5, 2, 6, 4)$ are left canonical Gauss codes and 
 ${\bf GLC}$, the set of left canobnical elements, is 
 $\{   (1, 3, 2, 6, 4, 5), (1, 3, 5, 2, 6, 4) \}$.

Now we consider unoriented equivalence classes.  

We first translate the map 
${\rm rev}^{\rm LP}: {\bf Gauss}^{\rm min, LP}(3) \to {\bf Gauss}^{\rm min, LP}(3)$ to 
a map 
$${\rm rev}^{\rm LP}: {\bf GLPmin} \to {\bf GLPmin}$$ as follows.  

Let $\pi$ be a permutation of $\{ 1, \dots, 6\}$ with 
 \begin{eqnarray*}
 \pi(1)=5, \quad \pi(2)= 6, \quad \pi(3)= 3, \quad \pi(4)= 4, \quad \pi(5)= 1, \quad \pi(6)= 2.  
  \end{eqnarray*}
In general, let $\pi$ be a permutation of $\{ 1, \dots, 2n\}$ with 
 \begin{eqnarray*}
\pi(k) = 2n-k \quad \mbox{(for odd $k$)} \quad \mbox{and} \quad \pi(k) = 2n+2  -k \quad 
\mbox{(for even $k$)}.  
 \end{eqnarray*}
For an element $w= (m_1, m_2, \dots, m_6)$,  let 
 \begin{eqnarray*}
 \pi_\ast({\rm rev}(w))= 
\pi_\ast(m_6, \dots, m_2, m_1) = (\pi(m_6), \dots, \pi(m_2), \pi(m_1)).   
\end{eqnarray*}
It is weakly left preferred, i.e., the subsequence consisting of odd numbers is $(1,3,5)$. 
Apply a rotation to it, and we have a left preferred element.  This is 
${\rm rev}^{\rm LP}(w)$.   

For example, for $w= (1, 3, 2, 6, 4, 5)$, 
 \begin{eqnarray*}
 \pi_\ast({\rm rev}(1, 3, 2, 6, 4, 5) )= 
\pi_\ast(5, 4, 6, 2, 3, 1) = (1, 4, 2, 6, 3, 5),    
\end{eqnarray*}
and 
${\rm rev}^{\rm LP}(1, 3, 2, 6, 4, 5) = (1, 4, 2, 6, 3, 5)$.   

This implies that $(1, 3, 2, 6, 4, 5)$ is unorientedly equivalent to $(1, 4, 2, 6, 3, 5)$.  
Therefore the unoriented equivalence class $[(1, 3, 2, 6, 4, 5)]_{\rm unori}$ 
of $(1, 3, 2, 6, 4, 5)$ is the union of 
$[(1, 3, 2, 6, 4, 5)]_{\rm ori}$ and $[(1, 4, 2, 6, 3, 5)]_{\rm ori}$.  

Let ${\bf GLPminUnori}$ be the set of unoriented equivalence classes of elements of ${\bf GLPmin}$. 
It consists of a single element, 
\begin{equation*}
\begin{array}{llcc}
(1) &  d_{3,1} & = & [(1, 3, 2, 6, 4, 5)]_{\rm unori}= [(1, 3, 2, 6, 4, 5)]_{\rm ori} \cup [(1, 4, 2, 6, 3, 5)]_{\rm ori}.  
\end{array}
\end{equation*} 
The left canonical representative of the class $[(1, 3, 2, 6, 4, 5)]_{\rm unori}$ is $(1, 3, 2, 6, 4, 5)$, since the smallest element of $[(1, 3, 2, 6, 4, 5)]_{\rm ori}$ is  $(1, 3, 2, 6, 4, 5)$ and that of 
$[(1, 4, 2, 6, 3, 5)]_{\rm ori}$ is $(1, 3, 5, 2, 6, 4)$.  

Therefore we see that ${\bf GLPminOri}$ ($={\bf Gauss}^{\rm min, LP}_{\rm ori}(3)$) consists of $2$ elements 
 \begin{eqnarray*}
d_{3,1}^+= [(1, 3, 2, 6, 4, 5)]_{\rm ori}  \quad \mbox{and} \quad 
d_{3,1}^-= [(1, 3, 5, 2, 6, 4)]_{\rm ori}, 
\end{eqnarray*}
which stand for  
 \begin{eqnarray*}
[(1,L) (2,L) (1,R) (3,R) (2,R) (3,L)]_{\rm ori}    \quad \mbox{and} \quad [(1,L)(2,L)(3,L)(1,R)(3,R)(2,R)]_{\rm ori},      
\end{eqnarray*} 
and the set ${\bf GLPminUnori}$ ( $={\bf Gauss}^{\rm min, LP}_{\rm unori}(3)$) consists of a single element  
 \begin{eqnarray*}
d_{3,1} = [(1, 3, 2, 6, 4, 5)]_{\rm unoori}, 
\end{eqnarray*}
which stands for 
 \begin{eqnarray*}
[(1,L) (2,L) (1,R) (3,R) (2,R) (3,L)]_{\rm unori}.    
\end{eqnarray*}

The symbol $d_{3,1}$ means that it is the first element of 
${\bf GLPminUnori}$ identified with  ${\bf Gauss}^{\rm min, LP}_{\rm unori}(3)$.  
The smallest representative of $d_{3,1}$  is $(1, 3, 2, 6, 4, 5)$.   
$d_{3,1}^+$ and $d_{3,1}^+$ are  elements of  
${\bf GLPminOri}$ identified with ${\bf Gauss}^{\rm min, LP}_{\rm ori}(3)$ 
such that $d_{3,1}^+$ is represented by the smallest element $(1, 3, 2, 6, 4, 5)$ of 
$d_{3,1}$, and $d_{3,1}^-$ is represented by ${\rm rev}^{\rm LP}(1, 3, 2, 6, 4, 5)$.  

}\end{example} 

 \begin{example}{\rm 
Let us consider a case of $n=4$.  
The order on $J$ is 
\begin{equation}
 (1, L) < (1, R) < (2, L) < (2, R) <  (3, L) < (3, R) <  (4, L) < (4, R).  
 \end{equation}
For simplicity, we put 
\begin{eqnarray*}
&& 1= (1, L) , \quad 2= (1, R),  \quad 3= (2, L), \quad 4= (2, R),  \\ 
&& 5=  (3, L), \quad 6=  (3, R) , \quad 7=  (4, L), \quad 8=  (4, R). 
\end{eqnarray*}
 In this notation, $w=(m_1, m_2, \dots, m_8)$ is left preferred if and only if 
 its subsequence consisting of odd numbers is $(1,3,5,7)$ and $m_1 =1$.  
 
 Let ${\bf GLP}$ $(= {\bf Gauss}^{\rm {LP}}(4))$ be the set of elements $(m_1, m_2, \dots, m_8)$ which are left preferred.  It consists of $840$ elements: 
\begin{eqnarray*}
&& (1, 2, 3, 4, 5, 6, 7, 8), \quad (1, 2, 3, 4, 5, 6, 8, 7), \quad (1, 2, 3, 4, 5, 7, 6, 8), \quad (1, 2, 3, 4, 5, 7, 8, 6), \\ 
&& (1, 2, 3, 4, 5, 8, 6, 7), \quad (1, 2, 3, 4, 5, 8, 7, 6), \quad (1, 2, 3, 4, 6, 5, 7, 8), \quad (1, 2, 3, 4, 6, 5, 8, 7), \\ 
&& (1, 2, 3, 4, 6, 8, 5, 7), \quad (1, 2, 3, 4, 8, 5, 6, 7), \quad (1, 2, 3, 4, 8, 5, 7, 6), \quad (1, 2, 3, 4, 8, 6, 5, 7), \\ 
&& \hspace{3cm} \dots  \dots \\
&& \hspace{3cm} \dots \dots \\
&& (1, 8, 6, 3, 4, 5, 2, 7), \quad (1, 8, 6, 3, 4, 5, 7, 2), \quad (1, 8, 6, 3, 5, 2, 4, 7), \quad (1, 8, 6, 3, 5, 2, 7, 4), \\ 
&& (1, 8, 6, 3, 5, 4, 2, 7), \quad (1, 8, 6, 3, 5, 4, 7, 2), \quad (1, 8, 6, 3, 5, 7, 2, 4), \quad (1, 8, 6, 3, 5, 7, 4, 2), \\ 
&& (1, 8, 6, 4, 2, 3, 5, 7), \quad (1, 8, 6, 4, 3, 2, 5, 7), \quad (1, 8, 6, 4, 3, 5, 2, 7), \quad (1, 8, 6, 4, 3, 5, 7, 2).
\end{eqnarray*}

In order to make a list of  ${\bf Gauss}^{\rm min, LP}(4)$, we need to remove elements which are 
$1$-reducible or $2$-reducible from ${\bf GLP}$. 

In the notation, $(m_1, m_2, \dots, m_8)$ is $1$-reducible if and only if 
at least of 
\begin{equation}
(m_1, m_2), (m_2, m_3), \dots, (m_7, m_8), (m_8, m_1)
 \end{equation} 
belongs to 
\begin{equation}
{\bf A1}:= \{ (1,2), (2,1), (3,4), (4,3), (5,6), (6,5), (7,8), (8,7) \}. 
 \end{equation}
 
 Let 
 \begin{equation}
{\bf PrepA2} = \{ (1, 2), (1, 3), (1,4), (2, 1), (2, 3), (2,4), (3, 1), (3, 2), (3,4), (4,1), (4,2), (4,3) \},  
 \end{equation}
which is the set of pairs $(i,j)$ with $i \neq j$ and $1\leq i, j \leq 4$. 
For an element $a \in {\bf PrepA2}$, we denote by $a[1]$ and $a[2]$ 
the first and second components of $a$. 

Let ${\bf A2}$ be a set consisting of all $( (p_1, p_2), (q_1, q_2) )$   with 
$ 1 \leq p_1, p_2, q_1, q_2 \leq 8$  
such that 
\begin{itemize}
\item[(i)] $( (p_1, p_2), (q_1, q_2) ) = ( (2a[1]-1, 2a[2]), (2a[1], 2a[2]-1)$ 
for some $a \in {\bf PrepA2}$, 
\item[(ii)] $( (p_1, p_2), (q_1, q_2) ) = ( (2a[1]-1, 2a[2]), (2a[2]-1, 2a[1])$ 
for some $a \in {\bf PrepA2}$, or 
\item[(iii)] $( (p_1, p_2), (q_1, q_2) ) = ( (2a[1], 2a[2]-1), (2a[2], 2a[1]-1)$ 
for some $a \in {\bf PrepA2}$.  
\end{itemize}
Then ${\bf A2}$ consists of 
\begin{eqnarray*}
&& ((1, 4), (2, 3)), \quad ((1, 4), (3, 2)), \quad ((1, 6), (2, 5)), \quad ((1, 6), (5, 2)), \quad ((1, 8), (2, 7)), \\ 
&& ((1, 8), (7, 2)), \quad ((2, 3), (4, 1)), \quad ((2, 5), (6, 1)), \quad ((2, 7), (8, 1)), \quad ((3, 2), (1, 4)), \\ 
&& ((3, 2), (4, 1)), \quad ((3, 6), (4, 5)), \quad ((3, 6), (5, 4)), \quad ((3, 8), (4, 7)), \quad ((3, 8), (7, 4)), \\ 
&& ((4, 1), (2, 3)), \quad ((4, 5), (6, 3)), \quad ((4, 7), (8, 3)), \quad ((5, 2), (1, 6)), \quad ((5, 2), (6, 1)), \\ 
&& ((5, 4), (3, 6)), \quad ((5, 4), (6, 3)), \quad ((5, 8), (6, 7)), \quad ((5, 8), (7, 6)), \quad ((6, 1), (2, 5)), \\ 
&& ((6, 3), (4, 5)), \quad ((6, 7), (8, 5)), \quad ((7, 2), (1, 8)), \quad ((7, 2), (8, 1)), \quad ((7, 4), (3, 8)), \\ 
&& ((7, 4), (8, 3)), \quad ((7, 6), (5, 8)), \quad ((7, 6), (8, 5)), \quad ((8, 1), (2, 7)), \quad ((8, 3), (4, 7)), \\ 
&& ((8, 5), (6, 7)).  
\end{eqnarray*}
 
 In the notion, $(m_1, m_2, \dots, m_8)$ is $2$-reducible if and only if 
 there are integers $i$ and $j$ in $\{1, \dots, 8\}$ such that 
 $( (m_i, m_{i+1}), (m_j, m_{j+1}) )$ belongs to ${\bf A2}$.  
 Here we assume $m_9= m_1$. 
 
 Let ${\bf GLPmin}$ $(={\bf Gauss}^{\rm min, LP}(4))$ be the set of elements 
$(m_1, m_2, \dots, m_8)$ which are left preferred and minimal.  It is 
obtained from ${\bf GLP}$ by removing all elements that are $1$-reducible or $2$-reducible.  
  
 The set ${\bf GLPmin}$ consists of $124$ elements, 
 \begin{eqnarray*}
&& (1, 3, 2, 4, 5, 7, 6, 8), \quad (1, 3, 2, 4, 6, 8, 5, 7), \quad (1, 3, 2, 5, 4, 7, 6, 8), \quad (1, 3, 2, 5, 4, 8, 6, 7), \\ 
&& (1, 3, 2, 5, 7, 4, 6, 8), \quad (1, 3, 2, 5, 7, 6, 4, 8), \quad (1, 3, 2, 5, 8, 6, 4, 7), \quad (1, 3, 2, 6, 4, 8, 5, 7), \\ 
&& (1, 3, 2, 6, 8, 4, 5, 7), \quad (1, 3, 2, 6, 8, 5, 4, 7), \quad (1, 3, 2, 8, 4, 5, 7, 6), \quad (1, 3, 2, 8, 5, 7, 4, 6), \\ 
&& \hspace{3cm} \dots  \dots \\
&& \hspace{3cm} \dots \dots \\
&& (1, 8, 3, 2, 6, 4, 5, 7), \quad (1, 8, 3, 5, 2, 4, 6, 7), \quad (1, 8, 3, 5, 4, 2, 6, 7), \quad (1, 8, 3, 5, 7, 6, 2, 4), \\ 
&& (1, 8, 3, 6, 4, 2, 5, 7), \quad (1, 8, 4, 2, 3, 5, 7, 6), \quad (1, 8, 4, 2, 6, 3, 5, 7), \quad (1, 8, 4, 6, 2, 3, 5, 7), \\ 
&& (1, 8, 4, 6, 3, 2, 5, 7), \quad (1, 8, 6, 2, 3, 5, 4, 7), \quad (1, 8, 6, 3, 5, 2, 4, 7), \quad (1, 8, 6, 4, 2, 3, 5, 7). 
\end{eqnarray*}

This is a complete list of ${\bf Gauss}^{\rm min, LP}(4)$.  
Now we consider oriented equivalence classes.  

For an element $w = (m_1, m_2, \dots, m_8)$ of ${\bf GLPmin}$, 
its oriented equivalence class is 
 \begin{eqnarray*}
 \{ w,  \quad  {\rm shift^{LP}}(w), \quad  ({\rm shift^{LP}})^2(w), \quad  ({\rm shift^{LP}})^3(w) \}.  
 \end{eqnarray*}
 
${\rm shift^{LP}}(w)$ is obtained from $(m_1 -2, m_2 -2 , \dots, m_8 -2)$ 
by rotation so that the first element becomes $1$.  Here 
$0$ and $-1$ in $(m_1 -2, m_2 -2 , \dots, m_8-2)$ 
are identified with $8$ and $7$ respectively. 

$({\rm shift^{LP}})^2(w)$ is obtained from $(m_1 -4, m_2 -4 , \dots, m_8 -4)$ 
by rotation so that the first element becomes $1$.   

$({\rm shift^{LP}})^3(w)$ is obtained from $(m_1 -6, m_2 -6 , \dots, m_8 -6)$ 
by rotation so that the first element becomes $1$.   

For example, for $w = (1, 3, 2, 4, 5, 7, 6, 8)$, 
${\rm shift^{LP}}(w) =(1, 8, 2, 3, 5, 4, 6, 7)$, $({\rm shift^{LP}})^2(w) =(1, 3, 2, 4, 5, 7, 6, 8)$ and $({\rm shift^{LP}})^3(w) =(1, 8, 2, 3, 5, 4, 6, 7)$.  The  oriented equivalence class 
$[w]_{\rm ori} = [(1, 3, 2, 4, 5, 7, 6, 8)]_{\rm ori}$  is 
 \begin{eqnarray*}
\{  (1, 3, 2, 4, 5, 7, 6, 8), \quad (1, 3, 2, 4, 5, 7, 6, 8), \quad (1, 8, 2, 3, 5, 4, 6, 7), \quad (1, 8, 2, 3, 5, 4, 6, 7) \}.  
 \end{eqnarray*}
 The left canonical representative of the class is $(1, 3, 2, 4, 5, 7, 6, 8)$. 

Let ${\bf GLPminOri}$ be the set of oriented equivalence classes 
of elements of ${\bf GLPmin}$.   It consists of $32$ elements, 
\begin{equation*}
\begin{array}{llcc}
(1) &   d_{4, 1}^+ &=& \{(1, 3, 2, 4, 5, 7, 6, 8), (1, 3, 2, 4, 5, 7, 6, 8), (1, 8, 2, 3, 5, 4, 6, 7), (1, 8, 2, 3, 5, 4, 6, 7)\}, \\ 
(2) &   d_{4, 2}^+ &=&  \{(1, 3, 2, 4, 6, 8, 5, 7), (1, 3, 5, 4, 6, 8, 2, 7), (1, 3, 5, 7, 6, 8, 2, 4), (1, 8, 2, 4, 6, 3, 5, 7)\}, \\ 
(3) &   d_{4, 3}^+ &=&  \{(1, 3, 2, 5, 4, 7, 6, 8), (1, 8, 2, 3, 5, 4, 7, 6), (1, 8, 3, 2, 4, 5, 7, 6), (1, 8, 3, 2, 5, 4, 6, 7)\}, \\ 
(4) &   d_{4, 4}^+ &=&  \{(1, 3, 2, 5, 4, 8, 6, 7), (1, 3, 5, 4, 7, 6, 2, 8), (1, 8, 3, 2, 6, 4, 5, 7), (1, 8, 4, 2, 3, 5, 7, 6)\}, \\ 
(5) &   d_{4, 5}^+ &=&  \{(1, 3, 2, 5, 7, 4, 6, 8), (1, 3, 8, 2, 4, 5, 7, 6), (1, 6, 8, 2, 3, 5, 4, 7), (1, 8, 3, 5, 2, 4, 6, 7)\}, \\ 
(6) &   d_{4, 6}^+ &=&  \{(1, 3, 2, 5, 7, 6, 4, 8), (1, 3, 2, 8, 4, 5, 7, 6), (1, 8, 3, 5, 4, 2, 6, 7), (1, 8, 6, 2, 3, 5, 4, 7)\}, \\ 
(7) &   d_{4, 7}^+ &=&  \{(1, 3, 2, 5, 8, 6, 4, 7), (1, 3, 5, 4, 7, 2, 8, 6), (1, 4, 2, 8, 3, 5, 7, 6), (1, 8, 3, 6, 4, 2, 5, 7)\}, \\ 
(8) &   d_{4, 8}^+ &=&  \{(1, 3, 2, 6, 4, 8, 5, 7), (1, 3, 5, 4, 8, 6, 2, 7), (1, 3, 5, 7, 6, 2, 8, 4), (1, 8, 4, 2, 6, 3, 5, 7)\}, \\ 
(9) &   d_{4, 9}^+ &=&  \{(1, 3, 2, 6, 8, 4, 5, 7), (1, 3, 5, 4, 8, 2, 6, 7), (1, 3, 5, 7, 6, 2, 4, 8), (1, 8, 4, 6, 2, 3, 5, 7)\}, \\ 
(10) & d_{4,10}^+ &=&  \{(1, 3, 2, 6, 8, 5, 4, 7), (1, 3, 5, 4, 8, 2, 7, 6), (1, 8, 3, 5, 7, 6, 2, 4), (1, 8, 4, 6, 3, 2, 5, 7)\}, \\ 
(11) & d_{4,11}^+ &=&  \{(1, 3, 2, 8, 5, 7, 4, 6), (1, 3, 8, 2, 5, 7, 6, 4), (1, 6, 8, 3, 5, 4, 2, 7), (1, 8, 6, 3, 5, 2, 4, 7)\}, \\ 
(12) & d_{4,12}^+ &=&  \{(1, 3, 2, 8, 6, 4, 5, 7), (1, 3, 5, 4, 2, 8, 6, 7), (1, 3, 5, 7, 6, 4, 2, 8), (1, 8, 6, 4, 2, 3, 5, 7)\}, \\ 
(13) & d_{4,13}^+ &=&  \{(1, 3, 5, 2, 4, 7, 6, 8), (1, 3, 8, 2, 5, 4, 6, 7), (1, 6, 8,  3, 2, 4, 5, 7), (1, 8, 2, 3, 5, 7, 4, 6)\}, \\ 
(14) & d_{4,14}^+ &=&  \{(1, 3, 5, 2, 4, 8, 6, 7), (1, 3, 5, 7, 4, 6, 2, 8), (1, 3, 8, 2, 6, 4, 5, 7), (1, 6, 8, 4, 2, 3, 5, 7)\}, \\ 
(15) & d_{4,15}^+ &=&  \{(1, 3, 5, 2, 6, 7, 4, 8), (1, 3, 8, 4, 5, 2, 6, 7), (1, 6, 2, 3, 5, 7, 4, 8), (1, 6, 2, 3, 8, 4, 5, 7)\}, \\ 
(16) & d_{4,16}^+ &=&  \{(1, 3, 5, 2, 6, 8, 4, 7), (1, 3, 5, 7, 4, 8, 2, 6), (1, 3, 8, 4, 6, 2, 5, 7), (1, 6, 2, 4, 8, 3, 5, 7)\}, \\ 
(17) & d_{4,10}^- &=&  \{(1, 3, 5, 2, 7, 6, 8, 4), (1, 3, 8, 5, 4, 6, 2, 7), (1, 6, 3, 2, 4, 8, 5, 7), (1, 8, 2, 6, 3, 5, 7, 4)\}, \\ 
(18) & d_{4,17}^+ &=&  \{(1, 3, 5, 2, 8, 4, 6, 7), (1, 3, 5, 7, 4, 2, 6, 8), (1, 3, 8, 6, 2, 4, 5, 7), (1, 6, 4, 8, 2, 3, 5, 7)\}, \\ 
(19) & d_{4,18}^+ &=&  \{(1, 3, 5, 2, 8, 6, 4, 7), (1, 3, 5, 7, 4, 2, 8, 6), (1, 3, 8, 6, 4, 2, 5, 7), (1, 6, 4, 2, 8, 3, 5, 7)\}, \\ 
(20) & d_{4,  7}^- &=&  \{(1, 3, 5, 2, 8, 6, 7, 4), (1, 3, 8, 6, 4, 5, 2, 7), (1, 6, 3, 5, 7, 4, 2, 8), (1, 6, 4, 2, 3, 8, 5, 7)\}, \\ 
(21) & d_{4,17}^- &=&  \{(1, 3, 5, 7, 2, 4, 8, 6), (1, 3, 5, 8, 2, 6, 4, 7), (1, 3, 6, 8, 4, 2, 5, 7), (1, 4, 6, 2, 8, 3, 5, 7)\}, \\
(22) & d_{4, 9}^- &=&  \{(1, 3, 5, 7, 2, 6, 8, 4), (1, 3, 5, 8, 4, 6, 2, 7), (1, 3, 6, 2, 4, 8, 5, 7), (1, 4, 8, 2, 6, 3, 5, 7)\}, \\ 
(23) & d_{4,14}^- &=&  \{(1, 3, 5, 7, 2, 8, 4, 6), (1, 3, 5, 8, 6, 2, 4, 7), (1, 3, 6, 4, 8, 2, 5, 7), (1, 4, 2, 6, 8, 3, 5, 7)\}, \\ 
(24) & d_{4,12}^- &=&  \{(1, 3, 5, 7, 2, 8, 6, 4), (1, 3, 5, 8, 6, 4, 2, 7), (1, 3, 6, 4, 2, 8, 5, 7), (1, 4, 2, 8, 6, 3, 5, 7)\}, \\ 
(25) & d_{4,13}^- &=&  \{(1, 3, 5, 8, 2, 7, 4, 6), (1, 3, 6, 8, 5, 2, 4, 7), (1, 4, 6, 3, 8, 2, 5, 7), (1, 6, 8, 3, 5, 7, 2, 4)\}, \\ 
(26) & d_{4,15}^- &=&  \{(1, 3, 5, 8, 4, 7, 2, 6), (1, 3, 6, 2, 5, 8, 4, 7), (1, 4, 8, 3, 5, 7, 2, 6), (1, 4, 8, 3, 6, 2, 5, 7)\}, \\ 
(27) & d_{4, 4}^- &=&  \{(1, 3, 5, 8, 6, 2, 7, 4), (1, 3, 6, 4, 8, 5, 2, 7), (1, 4, 2, 6, 3, 8, 5, 7), (1, 6, 3, 5, 7, 2, 8, 4)\}, \\ 
(28) & d_{4,19}^+ &=&  \{(1, 3, 6, 2, 5, 7, 4, 8), (1, 3, 8, 4, 5, 7, 2, 6), (1, 4, 8, 3, 5, 2, 6, 7), (1, 6, 2, 3, 5, 8, 4, 7)\}, \\ 
(29) & d_{4, 6}^-  &=&  \{(1, 3, 6, 2, 8, 5, 7, 4), (1, 3, 8, 5, 7, 2, 6, 4), (1, 4, 8, 6, 3, 5, 2, 7), (1, 6, 3, 5, 8, 4, 2, 7)\}, \\ 
(30) & d_{4, 5}^-  &=& \{(1, 3, 6, 8, 2, 5, 7, 4), (1, 3, 8, 5, 7, 2, 4, 6), (1, 4, 6, 8, 3, 5, 2, 7), (1, 6, 3, 5, 8, 2, 4, 7)\}, \\ 
(31) & d_{4, 3}^-  &=&  \{(1, 3, 6, 8, 5, 2, 7, 4), (1, 4, 6, 3, 8, 5, 2, 7), (1, 6, 3, 5, 8, 2, 7, 4), (1, 6, 3, 8, 5, 7, 2, 4)\}, \\ 
(32) & d_{4, 1}^-   &=&  \{(1, 3, 6, 8, 5, 7, 2, 4), (1, 3, 6, 8, 5, 7, 2, 4), (1, 4, 6, 3, 5, 8, 2, 7), (1, 4, 6, 3, 5, 8, 2, 7)\}. 
\end{array}
\end{equation*}

For each class listed above, the first element is the left canonical representative.  Thus we also have a complete list of ${\bf Gauss}^{\rm min, LC}(4)$.  

Now we consider unoriented equivalence classes.  

First we translate the map 
${\rm rev}^{\rm LP}: {\bf Gauss}^{\rm min, LP}(4) \to {\bf Gauss}^{\rm min, LP}(4)$ to 
a map 
$${\rm rev}^{\rm LP}: {\bf GLPmin} \to {\bf GLPmin}$$ as follows.  
Let $\pi$ be a permutation of $\{ 1, \dots, 8\}$ with 
 \begin{eqnarray*}
 \pi(1)=7, \quad \pi(2)= 8, \quad \pi(3)= 5, \quad \pi(4)= 6, \quad \pi(5)= 3, \quad \pi(6)= 4, 
 \quad \pi(7)= 1, \quad \pi(8)= 2.  
  \end{eqnarray*}
Recall that, in general, let $\pi$ be a permutation of $\{ 1, \dots, 2n\}$ with 
 \begin{eqnarray*}
\pi(k) = 2n-k \quad \mbox{(for odd $k$)} \quad \mbox{and} \quad \pi(k) = 2n+2  -k \quad 
\mbox{(for even $k$)}.  
 \end{eqnarray*}
For an element $w= (m_1, m_2, \dots, m_8)$,  let 
 \begin{eqnarray*}
 \pi_\ast({\rm rev}(w))= 
\pi_\ast(m_8, \dots, m_2, m_1) = (\pi(m_8), \dots, \pi(m_2), \pi(m_1)).   
\end{eqnarray*}
It is weakly left preferred, i.e., the subsequence consisting of odd numbers is $(1,3,5,7)$. 
Apply a rotation to it, and we have a left preferred element.  This is 
${\rm rev}^{\rm LP}(w)$.   

For example, for $w= (1, 3, 2, 4, 5, 7, 6, 8)$, 
 \begin{eqnarray*}
 \pi_\ast({\rm rev}(1, 3, 2, 4, 5, 7, 6, 8) ) = 
\pi_\ast(8, 6, 7, 5, 4, 2, 3, 1) = (2, 4, 1, 3, 6, 8, 5, 7),    
\end{eqnarray*}
and 
${\rm rev}^{\rm LP}(1, 3, 2, 4, 5, 7, 6, 8) = (1, 3, 6, 8, 5, 7, 2, 4)$.    

This implies that $(1, 3, 2, 4, 5, 7, 6, 8)$ is unorientedly equivalent to $(1, 3, 6, 8, 5, 7, 2, 4)$.  
Therefore the unoriented equivalence class $[(1, 3, 2, 4, 5, 7, 6, 8)]_{\rm unori}$ 
of $(1, 3, 2, 4, 5, 7, 6, 8)$ is the union of 
$[(1, 3, 2, 4, 5, 7, 6, 8)]_{\rm ori}$ and $[(1, 3, 6, 8, 5, 7, 2, 4)]_{\rm ori}$.  

Let ${\bf GLPminUnori}$ be the set of unoriented equivalence classes of elements of ${\bf GLPmin}$. 
It consists of $19$ elements,  
 \begin{eqnarray*}
d_{4,  1} & =~ [ (1, 3, 2, 4, 5, 7, 6, 8)]_{\rm unori} 
& =~ [ (1, 3, 2, 4, 5, 7, 6, 8)]_{\rm ori} \cup [ (1, 3, 6, 8, 5, 7, 2, 4)]_{\rm ori}, 
\\ 
d_{4,  2} & =~ [ (1, 3, 2, 4, 6, 8, 5, 7)]_{\rm unori} 
& =~ [ (1, 3, 2, 4, 6, 8, 5, 7)]_{\rm ori}, 
\\ 
d_{4,  3} & =~ [ (1, 3, 2, 5, 4, 7, 6, 8)]_{\rm unori} 
& =~ [ (1, 3, 2, 5, 4, 7, 6, 8)]_{\rm ori} \cup [ (1, 3, 6, 8, 5, 2, 7, 4)]_{\rm ori}, 
\\ 
d_{4,  4} & =~ [ (1, 3, 2, 5, 4, 8, 6, 7)]_{\rm unori} 
& =~ [ (1, 3, 2, 5, 4, 8, 6, 7)]_{\rm ori} \cup [ (1, 3, 5, 8, 6, 2, 7, 4)]_{\rm ori}, 
\\  
d_{4,  5} & =~ [ (1, 3, 2, 5, 7, 4, 6, 8)]_{\rm unori} 
& =~ [ (1, 3, 2, 5, 7, 4, 6, 8)]_{\rm ori} \cup [ (1, 3, 6, 8, 2, 5, 7, 4)]_{\rm ori}, 
\\  
d_{4,  6} & =~ [ (1, 3, 2, 5, 7, 6, 4, 8)]_{\rm unori} 
& =~ [ (1, 3, 2, 5, 7, 6, 4, 8)]_{\rm ori} \cup [ (1, 3, 6, 2, 8, 5, 7, 4)]_{\rm ori}, 
\\  
d_{4,  7} & =~ [ (1, 3, 2, 5, 8, 6, 4, 7)]_{\rm unori} 
& =~ [ (1, 3, 2, 5, 8, 6, 4, 7)]_{\rm ori} \cup [ (1, 3, 5, 2, 8, 6, 7, 4)]_{\rm ori}, 
\\  
d_{4,  8} & =~ [ (1, 3, 2, 6, 4, 8, 5, 7)]_{\rm unori} 
& =~ [ (1, 3, 2, 6, 4, 8, 5, 7)]_{\rm ori}, 
\\ 
d_{4,  9} & =~ [ (1, 3, 2, 6, 8, 4, 5, 7)]_{\rm unori} 
& =~ [ (1, 3, 2, 6, 8, 4, 5, 7)]_{\rm ori} \cup [ (1, 3, 5, 7, 2, 6, 8, 4)]_{\rm ori}, 
\\  
d_{4,  10} & =~ [ (1, 3, 2, 6, 8, 5, 4, 7)]_{\rm unori} 
& =~ [ (1, 3, 2, 6, 8, 5, 4, 7)]_{\rm ori} \cup [ (1, 3, 5, 2, 7, 6, 8, 4)]_{\rm ori}, 
\\  
d_{4,  11} & =~ [ (1, 3, 2, 8, 5, 7, 4, 6)]_{\rm unori} 
& =~ [ (1, 3, 2, 8, 5, 7, 4, 6)]_{\rm ori}, 
\\  
d_{4,  12} & =~ [ (1, 3, 2, 8, 6, 4, 5, 7)]_{\rm unori} 
& =~ [ (1, 3, 2, 8, 6, 4, 5, 7)]_{\rm ori} \cup [ (1, 3, 5, 7, 2, 8, 6, 4)]_{\rm ori}, 
\\  
d_{4,  13} & =~ [ (1, 3, 5, 2, 4, 7, 6, 8)]_{\rm unori} 
& =~ [ (1, 3, 5, 2, 4, 7, 6, 8)]_{\rm ori} \cup [ (1, 3, 5, 8, 2, 7, 4, 6)]_{\rm ori}, 
\\  
d_{4,  14} & =~ [ (1, 3, 5, 2, 4, 8, 6, 7)]_{\rm unori} 
& =~ [ (1, 3, 5, 2, 4, 8, 6, 7)]_{\rm ori} \cup [ (1, 3, 5, 7, 2, 8, 4, 6)]_{\rm ori}, 
\\   
d_{4,  15} & =~ [ (1, 3, 5, 2, 6, 7, 4, 8)]_{\rm unori} 
& =~ [ (1, 3, 5, 2, 6, 7, 4, 8)]_{\rm ori} \cup [ (1, 3, 5, 8, 4, 7, 2, 6)]_{\rm ori}, 
\\ 
d_{4,  16} & =~ [ (1, 3, 5, 2, 6, 8, 4, 7)]_{\rm unori} 
& =~ [ (1, 3, 5, 2, 6, 8, 4, 7)]_{\rm ori}, 
\\ 
d_{4,  17} & =~ [ (1, 3, 5, 2, 8, 4, 6, 7)]_{\rm unori} 
& =~ [ (1, 3, 5, 2, 8, 4, 6, 7)]_{\rm ori} \cup [ (1, 3, 5, 7, 2, 4, 8, 6)]_{\rm ori}, 
\\ 
d_{4,  18} & =~ [ (1, 3, 5, 2, 8, 6, 4, 7)]_{\rm unori} 
& =~ [ (1, 3, 5, 2, 8, 6, 4, 7)]_{\rm ori}, 
\\ 
d_{4,  19} & =~ [ (1, 3, 6, 2, 5, 7, 4, 8)]_{\rm unori} 
& =~ [ (1, 3, 6, 2, 5, 7, 4, 8)]_{\rm ori}.  
  \end{eqnarray*}

For each unoriented equivalence class listed above, the representative given there is the left canonical representative. Thus we also have a list of ${\bf Gauss}^{\rm min, LC}_{\rm rev}(4)$.  

Therefore we see that ${\bf Gauss}^{\rm min, LP}_{\rm ori}(4)$ consists of $32$ elements 
and the set  ${\bf Gauss}^{\rm min, LP}_{\rm unori}(4)$ consists of  $19$ elements. 
}\end{example} 

\section{Left canonical Gauss codes and orientations} \label{codes_orientations}

In this section we first summarize the results we have seen for virtual diagrams and virtual doodles. 
Then we discuss canonical orientations for unoriented virtual diagrams and unoriented virtual doodles.

Let $D$ be a virtual diagram with $n$ $(>0)$ real crossings.  
Let $w_0$ be a Gauss code of $D$ with respect to a base semiarc and a labeling of real crossings.  
Let $w = {\rm proj}^{\rm LP}(w_0)$, which is a left preferred Gauss code presenting $D$, 
where 
${\rm proj}^{\rm LP}: {\bf Gauss}(n) \to {\bf Gauss}^{\rm {LP}}(n)$ is a map defined in Section~\ref{section:leftpreferred}.  
Then $w$ is also a Gauss code of $D$. 

The oriented equivalence class of $w$ as left preferred Gauss codes is 
\begin{equation}
[w]^{\rm LP}_{\rm ori} =   
\{  
w, ~ 
{\rm shift^{LP}}(w), ~ 
({\rm shift^{LP}})^2(w), ~ 
\dots, 
({\rm shift^{LP}})^{n-1}(w)  
\}
\end{equation}
and the unoriented equivalence class of $w$ as left preferred Gauss codes is 
\begin{equation}
\begin{array}{ccl}
[w]^{\rm LP}_{\rm unori}   
  & = & [w]^{\rm LP}_{\rm ori} \cup [w']^{\rm LP}_{\rm ori}  \\
  & = &  
\{
w,~ {\rm shift^{LP}}(w),~ ({\rm shift^{LP}})^2(w),~  \dots,~ ({\rm shift^{LP}})^{n-1}(w)
\}
 \\
  &   &   \cup 
\{
w',~ {\rm shift^{LP}}(w'),~ ({\rm shift^{LP}})^2(w'),~  \dots,~ ({\rm shift^{LP}})^{n-1}(w')
\}, 
\end{array}
\end{equation}
where $w' = {\rm rev^{LP}}(w) = {\rm proj^{LP}}( {\rm rev}(w) )$.

\begin{definition}{\rm 
The {\it oriented left canonical Gauss code} of $D$, denoted by $G_{\rm ori}(D)$, 
is the smallest element of $[w]^{\rm LP}_{\rm ori}$, i.e.,  $G_{\rm ori}(D) = {\rm proj^{LC}}(w)$.  

The  {\it unoriented left canonical Gauss code} of $D$, denoted by $G_{\rm unori}(D)$, is 
the smallest element of $[w]^{\rm LP}_{\rm unori}$, i.e., it is the smaller one between 
${\rm proj^{LC}}(w)$ and ${\rm proj^{LC}}({\rm rev}(w))$.
}\end{definition} 

Theorem~\ref{thm:DGL} may be restated as follows.

\begin{theorem} \label{theorem:canonicalA}
Let $D$ and $D'$ be virtual diagrams with $n$ real crossings.  
\begin{itemize}
\item[(1)] 
$D$ is strictly equivalent to $D'$  if and only if 
$G_{\rm ori}(D)=G_{\rm ori}(D')$.  
\item[(2)] 
$D$ is strictly equivalent to $D'$ or ${\rm rev}(D')$  if and only if 
$G_{\rm unori}(D)=G_{\rm unori}(D')$.  
\end{itemize}
\end{theorem}

More precisely, we have seen the following. 

Let $w_0$ and $w'_0$ be Gauss codes presenting $D$ and $D'$, respectively,  and let 
$w = {\rm proj}^{\rm LP}(w_0)$ and $w'= {\rm proj}^{\rm LP}(w'_0)$.  
The following conditions are mutually equivalent. 
\begin{itemize}
\item[(i)]  $D$ is strictly equivalent to $D'$. 
\item[(ii)]  $w_0$ and $w'_0$ are orientedly equivalent as Gauss codes.
\item[(iii)] $w$ and $w'$ are orientedly equivalent as left preferred Gauss codes. 
\item[(iv)] $[w]^{\rm LP}_{\rm ori} = [w']^{\rm LP}_{\rm ori}$.  
\item[(v)] $G_{\rm ori}(D)=G_{\rm ori}(D')$. 
\end{itemize} 

By Proposition~\ref{propAB}, we have (i) $\Leftrightarrow$ (ii).  
By Theorem~\ref{theorem:leftA} (or Theorem~\ref{theorem:leftB}), we have (ii) $\Leftrightarrow$ (iii).  
(iii) $\Leftrightarrow$ (iv) $\Leftrightarrow$ (v) are obvious. 
Similarly we see that the 
following conditions are mutually equivalent. 
\begin{itemize}
\item[(i)]  $D$ is strictly equivalent to $D'$ or ${\rm rev}(D')$. 
\item[(ii)]  $w_0$ and $w'_0$ are unorientedly equivalent as Gauss codes.
\item[(iii)] $w$ and $w'$ are unorientedly equivalent as left preferred Gauss codes. 
\item[(iv)] $[w]^{\rm LP}_{\rm unori} = [w']^{\rm LP}_{\rm unori}$.  
\item[(v)] $G_{\rm unori}(D)=G_{\rm unori}(D')$. 
\end{itemize} 

Theorem~\ref{thm:DDG} may be restated as follows.

\begin{theorem} \label{theorem:canonicalB}
Let $D$ and $D'$ be minimal virtual diagrams with $n$ real crossings.  
\begin{itemize}
\item[(1)] 
$D$ is orientedly equivalent to $D'$  if and only if 
$G_{\rm ori}(D)=G_{\rm ori}(D')$.  
\item[(2)] 
$D$ is unorientedly equivalent to $D'$  if and only if 
$G_{\rm unori}(D)=G_{\rm unori}(D')$.  
\end{itemize}
\end{theorem}

\begin{definition}{\rm 
Let $K$ be an oriented (or unoriented) virtual doodle.  The  
{\it left canonical Gauss code} of $K$, denoted by $G_{\rm ori}(K)$ (or $G_{\rm unori}(K)$), 
is $G_{\rm ori}(D)$ (or $G_{\rm unori}(D)$) 
for a minimal virtual diagram $D$ representing $K$.  
}\end{definition}  

Then Theorem~\ref{thm:DDG} may be restated as follows. 

\begin{theorem}
Let $K$ and $K'$ be oriented (or unoriented) virtual doodles.  
$K = K'$  if and only if 
$G_{\rm ori}(K)=G_{\rm ori}(K')$   
(or $G_{\rm unori}(K)=G_{\rm unori}(K')$).  
\end{theorem}

Now we discuss canonical orientations for unoriented virtual diagrams and unoriented virtual doodles. 

Let $D$ be a virtual diagram and let $|D|$ be the unoriented virtual diagram by forgetting the orientation. 
Let $G_{\rm ori}(D)$ and $G_{\rm ori}({\rm rev}(D))$ be left canonical Gauss codes of $D$ and ${\rm rev}(D)$.  
If $G_{\rm ori}(D) \leq G_{\rm ori}({\rm rev}(D))$ then we say that $D$ has a {\it left canonical orientation} for $|D|$.  
When $G_{\rm ori}(D) \neq  G_{\rm ori}({\rm rev}(D))$, there is a unique left canonical orientation for $|D|$.  
When $G_{\rm ori}(D) = G_{\rm ori}({\rm rev}(D))$, there are two left canonical orientations for $|D|$.

Let $|K|$ be an unoriented virtual doodle.  Let $D$ be a minimal virtual diagram representing $|K|$.  
If $D$ has a left canonical orientation for $|D|$, then we  
say that the oriented virtual doodle $K$ represented by $D$ has a {\it left canonical  orientation} for $|K|$.  

\begin{proposition}
For any unoriented virtual doodle, there is a unique left canonical orientation.  
\end{proposition}

\begin{proof}
Let $D$ be a minimal virtual diagram representing an unoriented virtual doodle $|K|$. 
Let $K$ and $K'$ be the oriented virtual doodles represented by $D$ and ${\rm rev}(D)$ respectively.   
\begin{itemize}
\item If $G_{\rm ori}(D) < G_{\rm ori}({\rm rev}(D))$, then  $K$ has a left canonical orientation for $|K|$. 
\item If $G_{\rm ori}(D) > G_{\rm ori}({\rm rev}(D))$, then  $K'$ has a left canonical orientation for $|K|$.  
\item If $G_{\rm ori}(D) = G_{\rm ori}({\rm rev}(D))$, then $K = K'$ and it has a left canonical orientation for $|K|$. 
\end{itemize}
\end{proof}

\section{Arrow diagrams} \label{section:arrow} 

We discuss arrow diagrams for virtual doodles.  The notion of arrow diagrams for virtual doodles is analogous to the notion of Gauss diagrams for virtual knots. However the methods assigning orientations on arrows (or chords) are different.  

Let $S^1$ be the unit circle in the $xy$-plane. Let $n$ is a positive integer and 
let $P_1, \dots, P_{2n}$ be points of 
$S^1$ such that $P_i = (\cos{(i \pi/n + \pi/2n)}, \sin{(i \pi/n + \pi/2n)})$ for $i \in \{1, \dots, 2n\}$.  
We assume that $P_{2n+1}$ is $P_1$.  

An {\it arrow} is an oriented chord connecting two points of $P_1, \dots, P_{2n}$. An {\it arrow diagram} (with $n$ arrows) is the circle $S^1$ with $n$ arrows  
such that every point of  $P_1, \dots, P_{2n}$ is a head or tail of an arrow. 

Let $w= x_1 x_2 \dots x_{2n}$ be a Gauss code on $n$ letters. 
For each $i \in \{1, \dots, 2n\}$,  assign the $i$-th element $x_i$  of $w$ to the point $P_i$. 
For each $j \in \{1, \dots, n\}$, there is a pair of points in $P_1, \dots, P_{2n}$ 
to which $(j,L)$ and $(j,R)$ are assigned. Connect the two points by a chord, labeled $j$, 
and give an orientation to this chord from the point assigned $(j,R)$ to the point assigned $(j,L)$. 
Then we have an arrow diagram with $n$ arrows such that the arrows are labeled with integers $1, \dots, n$. We call it the {\it arrow diagram with labeled arrows}  
of the Gauss code $w$. (We may forget the $J$-labels assigned to $P_1, \dots, P_{2n}$, since they are recovered from the labels of arrows.) 
By forgetting the labels on arrows,  we have an arrow diagram without labels, which we call 
the {\it arrow diagram} of the Gauss code $w$.

\begin{lemma}\label{lemmaDA}
Let $w$ and $w'$ be Gauss codes with $n$ letters and let $A$ and $A'$ be the 
arrow diagrams of them. 
\begin{itemize}
\item[(1)] $w' = \pi_\ast(w)$ for a permutation $\pi$ of $\{1, \dots, n\}$ if and only if $A=A'$.   
\item[(2)] $w' = {\rm shift}[m](w)$ for some integer $m$ if and only if $A'$ is obtained from $A$ by rotation by $-m \pi /n$ radian. 
\item[(3)] $w' = {\rm rev}(w)$ if and only if $A'$ is obtained from $A$ by reflection along the $x$-axis. 
\end{itemize}
\end{lemma}

\begin{proof} 
It is a direct consequence from the definition. 
\end{proof}

Let ${\bf Arrow}(n)$ be the set of arrow diagrams with $n$ arrows. 

We say that an arrow diagram $A'$ is obtained from an arrow diagram $A$ by a {\it reflection} if 
there is a line through the origin of the $xy$-plane along which the reflection changes $A$ into $A'$.  

The dihedral group $D_{2n}$ with $4n$ elements acts on ${\bf Arrow}(n)$ by rotations around the origin of the $xy$-plane by 
$m\pi/n$ radian rotation for $m \in \{0,1, \dots, 2n-1\}$ 
and $2n$ reflections along lines through the origin. The action is generated by 
the rotation by $\pi /n$ radian and the reflection along the $x$-axis.  

Let ${\bf Arrow}_{\rm rot}(n)$ be the set of arrow diagrams with $n$ arrows modulo the rotations, and 
${\bf Arrow}_{\rm rot+ref}(n)$ the set of arrow diagrams with $n$ arrows modulo the rotations and the reflections.  Namely, ${\bf Arrow}_{\rm rot}(n)$ (or ${\bf Arrow}_{\rm rot+ref}(n)$) is the quotient of ${\bf Arrow}(n)$ by the cyclic action by the rotations (or by the whole actions of the dihedral group.) 

We have the following. 

\begin{proposition}\label{propDA}
There are bijections:  
\begin{equation}
{\bf Diagram}_{\rm strict}(n)   \longleftrightarrow  
{\bf Gauss}_{\rm ori}(n) \longleftrightarrow  
{\bf Arrow}_{\rm rot}(n)
\end{equation} 
and  
\begin{equation}
{\bf Diagram}_{\rm strict+rev}(n)  \longleftrightarrow  
{\bf Gauss}_{\rm unori}(n)  \longleftrightarrow  
{\bf Arrow}_{\rm rot+ref}(n).
\end{equation}
\end{proposition} 

\begin{proof} 
This is a consequence of Proposition~\ref{propAB} and Lemma~\ref{lemmaDA}.  
\end{proof}

For the points $P_1, \dots, P_{2n}$ as before, we assume that $P_{2n+1}$ is $P_1$. 

An arrow diagram is {\it $1$-reducible} if there is an arrow connecting $P_i$ and $P_{i+1}$ for some $i 
\in \{1, \dots, 2n\}$.  Otherwise, it is called {\it $1$-irreducible}. 

An arrow diagram is {\it $2$-reducible} if one of the following holds: 
\begin{itemize}
\item[(1)] There are integers $i, i' \in \{1, \dots, 2n\}$ such that 
$P_i$ and $P_i'$ are connected by an arrow and $P_{i+1}$ and $P_{i'+1}$ are 
connected by an arrow such that if $P_i$ is a head (tail) then 
$P_{i+1}$ is a tail (head).  
\item[(2)] There are integers $i, i' \in \{1, \dots, 2n\}$ such that 
$P_i$ and $P_{i'+1}$ are connected by an arrow and $P_{i+1}$ and $P_{i'}$ are 
connected by an arrow such that if $P_i$ is a head (tail) then 
$P_{i+1}$ is a tail (head).  
\end{itemize} 
Otherwise, it is called {\it $2$-irreducible}. 

An arrow diagram is called {\it minimal} or {\it irreducible} if it is $1$-irreducible and 
$2$-irreducible. 

\begin{lemma}\label{lemmaArrowC}
Let $D$, $w$ and $A$ be a virtual diagram, a Gauss code of $D$ and the arrow diagram of $w$. 
The following three conditions are mutually equivalent: (1)  
$D$ is minimal, (2) $w$ is minimal, and (3) $A$ is minimal. 
\end{lemma}

\begin{proof}
It is a direct consequence from the definition. 
\end{proof}

Let ${\bf Arrow}_{\rm rot}^{\rm min}(n)$ (${\bf Arrow}_{\rm rot+ref}^{\rm min}(n)$)
denote the subset of 
${\bf Arrow}_{\rm rot}(n)$ (${\bf Arrow}_{\rm rot+ref}(n)$) consisting of 
the classes of minimal arrow diagrams with $n$ arrows. 

\begin{theorem}\label{propArrowDB}
There are bijections  
\begin{equation}
{\bf Doodle}_{\rm ori}(n)  \longleftrightarrow  
{\bf Diagram}^{\rm min}_{\rm strict}(n)  \longleftrightarrow  {\bf Gauss}^{\rm min}_{\rm ori}(n)  \longleftrightarrow  {\bf Arrow}^{\rm min}_{\rm rot}(n)
\end{equation}
and 
\begin{equation}
{\bf Doodle}_{\rm unori}(n)  \longleftrightarrow 
{\bf Diagram}^{\rm min}_{\rm strict+rev}(n)  \longleftrightarrow  {\bf Gauss}^{\rm min}_{\rm unori}(n)  \longleftrightarrow  {\bf Arrow}^{\rm min}_{\rm rot+ref}(n). 
\end{equation} 
\end{theorem}

\begin{proof}
Combining Theorem~\ref{thm0}, Proposition~\ref{propDB} and Proposition~\ref{propDA}, we have the result.  
\end{proof}

Computation of virtual doodles using arrow diagrams is easy and practical when the number of real crossings is small. 
At the conference \lq\lq Self-distributive system and quandle (co)homology theory in algebra and low-dimensional topology\rq\rq  held in Busan, Korea in June 2017 as 2017 KIAS Research Station, after the talk on doodles given by the fourth author, Victoria Lebed made a table of minimal arrow diagrams with $4$ arrows by hand.  By Theorem~\ref{propArrowDB} it provides a table of virtual doodles with $4$ real crossings. 

\vspace{0.5cm} 

\noindent
{\bf Acknowledgments} 

\noindent
The third and the fourth authors were supported by JSPS KAKENHI Grant Numbers JP15K04879 and JP26287013.

\end{document}